\UseRawInputEncoding 
\documentclass[12pt]{amsart}
\usepackage{mathrsfs}
\usepackage{}
\usepackage{amsmath}
\usepackage{amsfonts}
\usepackage{amssymb}
\usepackage[all,cmtip]{xy}           
\usepackage{shuffle}
\usepackage{caption}
\usepackage{bbding}
\usepackage{txfonts}
\usepackage[shortlabels]{enumitem}
\usepackage{ifpdf}
\ifpdf
\usepackage[colorlinks,final,backref=page,hyperindex]{hyperref}
\else
\usepackage[colorlinks,final,backref=page,hyperindex,hypertex]{hyperref}
\fi
\usepackage{tikz}
\usepackage[active]{srcltx}

\makeatletter

\newtheorem{theorem}{Theorem}[section]
\newtheorem{prop}[theorem]{Proposition}
\newtheorem{lemma}[theorem]{Lemma}
\newtheorem{coro}[theorem]{Corollary}
\newtheorem{prop-def}{Proposition-Definition}[section]

\theoremstyle{definition}
\newtheorem{defn}[theorem]{Definition}

\newtheorem{remark}[theorem]{Remark}
\newtheorem{exam}[theorem]{Example}

\newcommand{\nc}{\newcommand}

\newcommand {\emptycomment}[1]{}

\nc{\delete}[1]{{}}
\nc{\mmargin}[1]{}

\nc{\mlabel}[1]{\label{#1}}  
\nc{\mcite}[1]{\cite{#1}}  
\nc{\mref}[1]{\ref{#1}}  
\nc{\meqref}[1]{\eqref{#1}}  
\nc{\mbibitem}[1]{\bibitem{#1}} 

\newcommand{\g}{\mathfrak g}
\newcommand{\h}{\mathfrak h}

\newcommand{\bk}{{\mathbf{k}}}

\nc{\vep}{\varepsilon}
\nc{\btl}{\blacktriangleright}
\nc{\bin}[2]{ (_{\stackrel{\scs{#1}}{\scs{#2}}})}  
\nc{\binc}[2]{(\!\! \begin{array}{c} \scs{#1}\\
		\scs{#2} \end{array}\!\!)}  
\nc{\bincc}[2]{  ( {\scs{#1} \atop
		\vspace{-1cm}\scs{#2}} )}  
\nc{\oline}[1]{\overline{#1}}
\nc{\mapm}[1]{\lfloor\!|{#1}|\!\rfloor}
\nc{\bs}{\bar{S}}
\nc{\cast}{{\,\mbox{\raisebox{.8pt}{$\scriptstyle \circledast$}}\,}}
\nc{\la}{\longrightarrow}
\nc{\ot}{\otimes}
\nc{\rar}{\rightarrow}
\nc{\lda}{\leftharpoonup\!\!\!\!\!\leftharpoonup}
\nc{\rda}{\rightharpoonup\!\!\!\!\!\rightharpoonup}
\nc{\dar}{\downarrow}
\nc{\dap}[1]{\downarrow \rlap{$\scriptstyle{#1}$}}
\nc{\defeq}{\stackrel{\rm def}{=}}
\nc{\dis}[1]{\displaystyle{#1}}
\nc{\dotcup}{\ \displaystyle{\bigcup^\bullet}\ }
\nc{\hcm}{\ \hat{,}\ }
\nc{\hts}{\hat{\otimes}}
\nc{\hcirc}{\hat{\circ}}
\nc{\lleft}{[}
\nc{\lright}{]}
\nc{\curlyl}{\left \{ \begin{array}{c} {} \\ {} \end{array}
	\right .  \!\!\!\!\!\!\!}
\nc{\curlyr}{ \!\!\!\!\!\!\!
	\left . \begin{array}{c} {} \\ {} \end{array}
	\right \} }
\nc{\longmid}{\left | \begin{array}{c} {} \\ {} \end{array}
	\right . \!\!\!\!\!\!\!}
\nc{\ora}[1]{\stackrel{#1}{\rar}}
\nc{\ola}[1]{\stackrel{#1}{\la}}
\nc{\scs}[1]{\scriptstyle{#1}} \nc{\mrm}[1]{{\rm #1}}
\nc{\dirlim}{\displaystyle{\lim_{\longrightarrow}}\,}
\nc{\invlim}{\displaystyle{\lim_{\longleftarrow}}\,}
\nc{\dislim}[1]{\displaystyle{\lim_{#1}}} \nc{\colim}{\mrm{colim}}
\nc{\mvp}{\vspace{0.3cm}} \nc{\tk}{^{(k)}} \nc{\tp}{^\prime}
\nc{\ttp}{^{\prime\prime}} \nc{\svp}{\vspace{2cm}}
\nc{\vp}{\vspace{8cm}}
\nc{\modg}[1]{\!<\!\!{#1}\!\!>}
\nc{\intg}[1]{F_C(#1)}
\nc{\lmodg}{\!<\!\!}
\nc{\rmodg}{\!\!>\!}
\nc{\cpi}{\widehat{\Pi}}
\nc{\labs}{\mid\!}
\nc{\rabs}{\!\mid}
\nc{\btr}{\blacktriangleright}

\nc{\ad}{\mrm{ad}}
\nc{\rRB}{\mathsf{rRB}}
\nc{\cocrRB}{\mathsf{cocrRB}}
\nc{\coa}{\mrm{Coalg}}
\nc{\PH}{\mathsf{PH}}
\nc{\cocPH}{\mathsf{cocPH}}
\nc{\ann}{\mrm{ann}}
\nc{\Ad}{\mrm{Coad}}
\nc{\Aut}{\mrm{Aut}}
\nc{\Der}{\mrm{Der}}
\nc{\Sym}{\mrm{Sym}}
\nc{\br}{\mrm{bre}}
\nc{\can}{\mrm{can}}
\nc{\Cont}{\mrm{Cont}}
\nc{\rchar}{\mrm{char}}
\nc{\cok}{\mrm{coker}}
\nc{\de}{\mrm{dep}}
\nc{\dtf}{{R-{\rm tf}}}
\nc{\dtor}{{R-{\rm tor}}}

\nc{\Dif}{\mrm{Diff}}
\nc{\Div}{\mrm{Div}}
\nc{\End}{\mrm{End}}
\nc{\Ext}{\mrm{Ext}}
\nc{\Fil}{\mrm{Fil}}
\nc{\Fr}{\mrm{Fr}}
\nc{\Frob}{\mrm{Frob}}
\nc{\Gal}{\mrm{Gal}}
\nc{\GL}{\mrm{GL}}
\nc{\Gr}{\mrm{Gr}}
\nc{\Hom}{\mrm{Hom}}
\nc{\Hoch}{\mrm{Hoch}}
\nc{\hsr}{\mrm{H}}
\nc{\hpol}{\mrm{HP}}
\nc{\id}{\mrm{id}}
\nc{\im}{\mrm{im}}
\nc{\inv}{\mrm{inv}}
\nc{\Id}{\mrm{Id}}
\nc{\ID}{\mrm{ID}}
\nc{\Irr}{\mrm{Irr}}
\nc{\incl}{\mrm{incl}}
\nc{\length}{\mrm{length}}
\nc{\NLSW}{\mrm{NLSW}}
\nc{\Lie}{\mrm{Lie}}
\nc{\mchar}{\rm char}
\nc{\mpart}{\mrm{part}}
\nc{\ql}{{\QQ_\ell}}
\nc{\qp}{{\QQ_p}}
\nc{\rank}{\mrm{rank}}
\nc{\rcot}{\mrm{cot}}
\nc{\rdef}{\mrm{def}}
\nc{\rdiv}{{\rm div}}
\nc{\rtf}{{\rm tf}}
\nc{\rtor}{{\rm tor}}
\nc{\res}{\mrm{res}}
\nc{\SL}{\mrm{SL}}
\nc{\Spec}{\mrm{Spec}}
\nc{\tor}{\mrm{tor}}
\nc{\Tr}{\mrm{Tr}}
\nc{\tr}{\mrm{tr}}
\nc{\wt}{\mrm{wt}}

\nc{\bfk}{{\bf k}}
\nc{\bfone}{{\bf 1}}
\nc{\bfzero}{{\bf 0}}
\nc{\detail}{\marginpar{\bf More detail}
	\noindent{\bf Need more detail!}
	\svp}
\nc{\gap}{\marginpar{\bf Incomplete}\noindent{\bf Incomplete!!}
	\svp}
\nc{\FMod}{\mathbf{FMod}}
\nc{\Int}{\mathbf{Int}}
\nc{\Mon}{\mathbf{Mon}}
\nc{\remarks}{\noindent{\bf Remarks: }}
\nc{\Rep}{\mathbf{Rep}}
\nc{\Rings}{\mathbf{Rings}}
\nc{\Sets}{\mathbf{Sets}}
\nc{\Diff}{\mathbf{Diff}}
\nc{\Inte}{\mathbf{Inte}}
\nc{\U}{\mathrm{U}}
\newcommand{\YD}{{\mathcal{Y}\mathcal{D}}}

\nc{\BA}{{\mathbb A}}   \nc{\CC}{{\mathbb C}}
\nc{\DD}{{\mathbb D}}   \nc{\EE}{{\mathbb E}}
\nc{\FF}{{\mathbb F}}   \nc{\GG}{{\mathbb G}}
\nc{\HH}{{\mathbb H}}   \nc{\LL}{{\mathbb L}}
\nc{\NN}{{\mathbb N}}   \nc{\PP}{{\mathbb P}}
\nc{\QQ}{{\mathbb Q}}   \nc{\RR}{{\mathbb R}}
\nc{\TT}{{\mathbb T}}   \nc{\VV}{{\mathbb V}}
\nc{\ZZ}{{\mathbb Z}}   \nc{\TP}{\widetilde{P}}

\nc{\cala}{{\mathcal A}}    \nc{\calb}{{\mathcal B}}
\nc{\calc}{{\mathcal C}}
\nc{\cald}{\mathcal{D}}     \nc{\cale}{{\mathcal E}}
\nc{\calf}{{\mathcal F}}    \nc{\calg}{{\mathcal G}}
\nc{\calh}{{\mathcal H}}    \nc{\cali}{{\mathcal I}}
\nc{\call}{{\mathcal L}}    \nc{\calm}{{\mathcal M}}
\nc{\caln}{{\mathcal N}}    \nc{\calo}{{\mathcal O}}
\nc{\calp}{{\mathcal P}}    \nc{\calr}{{\mathcal R}}

\nc{\cals}{{\mathcal S}}    \nc{\calt}{{\Omega}}
\nc{\calv}{{\mathcal V}}    \nc{\calw}{{\mathcal W}}
\nc{\calx}{{\mathcal X}}

\nc{\fraka}{{\mathfrak a}}
\nc{\frakb}{\mathfrak{b}}
\nc{\frakg}{{\frak g}}
\nc{\frakl}{{\frak l}}
\nc{\fraks}{{\frak s}}
\nc{\frakB}{{\frak B}}
\nc{\frakm}{{\frak m}}
\nc{\frakM}{{\frak M}}
\nc{\frakp}{{\frak p}}
\nc{\frakW}{{\frak W}}
\nc{\frakX}{{\frak X}}
\nc{\frakS}{{\frak S}}
\nc{\frakA}{{\frak A}}
\nc{\frakx}{{\frakx}}

\nc{\ynr}[1]{\textcolor{orange}{\underline{Yunnan:}#1 }}

\nc{\lir}[1]{\textcolor{red}{\underline{Li:}#1 }}

\begin{document}

\title[Double cross products with projections]{Double cross products with projections and relative Rota-Baxter operators on Hopf algebras}

\author{Yunnan Li}
\address{School of Mathematics and Information Science, Guangzhou University,
Guangzhou 510006, China}
\email{ynli@gzhu.edu.cn}

\begin{abstract}
Given a matched pair of Hopf algebras, the double cross product construction yields a new Hopf algebra, the prototypical example of which is the generalized quantum double arising from a (Hopf) skew pairing. In this paper, we specifically investigate double cross products with projections in the framework of Radford's biproduct theory,
and establish that such an algebraic structure is equivalent to a simplified version of (Yetter-Drinfeld) relative Rota-Baxter operators on Hopf algebras recently introduced by Sciandra, with the corresponding braided Hopf algebras lying in the Yetter-Drinfeld module category under coadjoint coactions.
We further examine conditions under which this algebraic structure induces a matched pair of actions on a single Hopf algebra, which is in turn equivalent to the notion of braiding operators on Hopf algebras that yield solutions of the braid equation. In addition, we show that such relative Rota-Baxter operators are morphisms between certain Doi-Hopf modules.
\end{abstract}

\keywords{double cross product, matched pair, Yetter-Drinfeld module, relative Rota-Baxter operator
\\
\qquad 2020 Mathematics Subject Classification. 16T05, 17B38, 18M15}

\maketitle

\tableofcontents

\allowdisplaybreaks

\section{Introduction}

The concept of a matched pair of Hopf algebras was first introduced by Singer in the graded case in 1972. This idea was later extended to the ungraded case by Takeuchi in \cite{Ta}, where he also introduced the notion of a matched pair of groups. Subsequently, Majid presented the modern definition of a matched pair of Hopf algebras \cite{Ma3,Ma}, imposing an extra cocommutativity condition to guarantee the double cross product bialgebra structure (also called a bicrossproduct).
A typical example of double cross products is the generalized quantum double construction via (Hopf) skew pairings, which includes the classical Drinfeld double of a finite dimensional Hopf algebra.

In \cite{AGV} Angiono, Galindo, and Vendramin  introduced the notion of Hopf braces, and demonstrated that cocommutative Hopf braces are equivalent to matched pairs of cocommutative Hopf algebras with compatible actions, yielding Yang-Baxter operators.
Removing the cocommutativity hypothesis from \cite{AGV}, Ferri and Sciandra \cite{FS} recently introduced the notion of a matched pair of actions, defined as a matched pair of a Hopf algebra and itself equipped with compatible actions, and ingeniously discovered its equivalent structure, the so-called Yetter-Drinfeld brace. This generalizes the construction of braided groups due to Lu, Yan and Zhu~\cite{LYZ}, as well as cocommutative Hopf braces in \cite{AGV}. Also, it was pointed out in \cite{GGV1} that a braiding operator on a Hopf algebra $H$, equivalently, a matched pair of actions on $H$ or the induced Yetter-Drinfeld brace $(H_\rightharpoonup,H)$, provides a set-theoretic type solution of the braid equation. Later, we showed in \cite{Li} that a double cross product $H\bowtie H$ of a Hopf algebra $H$ and itself, equipped with a projection onto $H$, gives another structure equivalent to the aforementioned ones. Consequently, it is natural to ask, in general, what algebraic structure arises from a double cross product $K\bowtie H$ of two Hopf algebras $K$ and $H$ with a projection $\theta$ onto $H$. This is the main problem we address in the present paper.

On the other hand, Rota-Baxter operators (of weight 0) on a Lie algebra $\g$ were introduced by  Semenov-Tian-Shansky in~\cite{STS} as the operator form of tensor solutions of the classical Yang-Baxter equation on $\g$, while Kupershmidt~\cite{Ku} formulated the relative version in the presence of a representation of $\g$, now commonly referred to as a relative Rota-Baxter operator or an $O$-operator. Since then, relative Rota-Baxter operators have been defined for various algebraic structures, such as associative algebras~\cite{Uc}, Leibniz algebras~\cite{TS} and (Lie) groups~\cite{JSZ}.

The notion of relative Rota-Baxter operators on (cocommutative) Hopf algebras was originally introduced by the present author together with Sheng and Tang in \cite{LST}. It generalizes Goncharov's Rota-Baxter operators on cocommutative Hopf algebras~\cite{Go}, inspired by the innovative work on Rota-Baxter groups~\cite{GLS}. Subsequently, Sciandra \cite{Sc} proposed the notion of
Yetter-Drinfeld relative Rota-Baxter operators on an arbitrary Hopf algebra $H$ with respect to a bialgebra in the category ${_H^H}\YD$ of Yetter-Drinfeld modules over $H$, rather than with respect to a cocommutative $H$-module bialgebra as in \cite{LST}.
Since then, numerous works have investigated relative Rota-Baxter operators on Hopf algebras from various perspectives; see e.g.~\cite{VRP,VRP1,ZD}.

In this paper, we first show that a double cross product $K\bowtie H$ with a projection $\theta$ onto $H$ is equivalently characterized by a matched pair of Hopf algebras $(H,K,\rightharpoonup,\leftharpoonup)$ together with a Hopf algebra homomorphism $\pi:K\to H$
satisfying a crucial commutativity condition. In the special case of a matched pair of actions on $H$, this reduces to be the braided-commutative algebra condition of $H$ with respect to the corresponding braiding operator.
Furthermore, by Radford's biproduct (or Majid's bosonization) theory, it is well known that the subalgebra of coinvariants of $\theta$ in $K\bowtie H$ is a braided Hopf algebra in the category ${_H^H}\YD$.
Of particular interest is the fact that this braided object is isomorphic to $K$ as coalgebras. Consequently, we can transmute the Hopf algebra $K$ into a braided Hopf algebra in ${_H^H}\YD$, denoted by $K_\pi$, while preserving its original coalgebra structure.
Such a framework generalizes Majid's transmutation process for coquasitriangular Hopf algebras~\cite{Ma1}.

Next, for these transmutations $K_\pi$, we find that the maps $\pi:K_\pi \to H$ give rise to a class of (Yetter-Drinfeld) relative Rota-Baxter operators, where the coaction of $H$ on $K_\pi$ is given by the coadjoint coaction. Accordingly, we introduce a suitable class of relative Rota-Baxter operators on Hopf algebras, two subclasses of which are the original one introduced in \cite[Definition~3.1]{LST} under the cocommutativity restriction, and the one mentioned in \cite[Proposition~4.12]{Sc} under the bijectivity constraint. Notably, this is precisely the notion equivalent to double cross products with projections.

Finally, we investigate the condition under which a pair of actions $(\rightharpoonup_\pi,\leftharpoonup_\pi)$ on $K$ induced by $\pi:K\to H$ constitutes a matched pair of actions on $K$. Meanwhile, we prove that the map $\pi:(K,\leftharpoonup_\pi,(\pi\otimes\id_K)\Delta_K)\to(H,\leftharpoonup,\Delta_H)$ is exactly a morphism of Doi-Hopf modules.

The paper is organized as follows. In Section~\ref{sec:mp}, we recall the main notions used throughout the paper, including matched pairs of Hopf algebras and double cross products, (pre-)braided tensor categories of Yetter-Drinfeld modules with braided Hopf algebras therein, as well as Doi-Hopf modules.
In Section~\ref{sec:proj}, we first establish a fundamental criteria for when a matched pair of Hopf algebras gives rises to a double cross product with a projection~(Proposition~\ref{prop:dcp-hopf-map}).
We then apply Radford's biproduct theory to a double cross product $K\bowtie H$ with a projection $\theta$ onto $H$, and transplant the braided Hopf algebra structure of its subalgebra of coinvariants of $\theta$ in ${_H^H}\YD$ to the coalgebra $K$~(Theorem~\ref{thm:dcp-yd-hopf}). As typical examples of double cross products, we study the generalized quantum doubles to illustrate this framework, with particular emphasis on (co)quasitriangular Hopf algebras.
In Section~\ref{sec:rrb}, we first introduce an appropriate notion of relative Rota-Baxter operators on Hopf algebras, and then prove that they can be induced by double cross products with projections~(Theorem~\ref{thm:dcp-rrb}). Conversely, we show that such relative Rota-Baxter operators on Hopf algebras give rise to double cross products with projections~(Theorem~\ref{thm:rrb-dcp}).
In Section~\ref{sec:mpa}, we investigate the condition under which a double cross product with a projection induces a matched pair of actions on a single Hopf algebra~(Theorem~\ref{thm:dcp-mpa}). Also, we show that the relative Rota-Baxter operators arisen from double cross products with projections turn out to be morphisms of Doi-Hopf modules~(Theorem~\ref{thm:doi-hopf} and Corollary~\ref{coro:doi-hopf}).

\vspace{2mm}

\noindent
{\bf Convention.}
In this paper, we fix an arbitrary ground field $\bk$.
All the objects under discussion, including vector spaces, algebras and tensor products, are taken over $\bk$ by default.

For any unital algebra $(A,\mu,u)$ with multiplication $\mu$ and the unit map $u:\bk\to A$, let $\mu^{(0)}=\id$ and for $n\geq1$ we write
$$\mu^{(n)}=(\mu\otimes \id^{\otimes(n-1)})\cdots (\mu\otimes\id)\mu.$$

For any coalgebra $(C,\Delta,\vep)$, we compress the Sweedler notation of the comultiplication $\Delta$ as $$\Delta(x)=x_1\otimes x_2$$ for simplicity.
Furthermore, let $\Delta^{(0)}=\id$ and for $n\geq1$ we write
$$\Delta^{(n)}(x)=(\Delta\otimes\id^{\otimes (n-1)})\cdots(\Delta\otimes\id)\Delta(x)=x_1\otimes\cdots\otimes x_{n+1}.$$

By convention, a Hopf algebra is denoted by $H=(H,\cdot\,,1,\Delta,\vep,S)$.
Denote by $G(H)$ the set of group-like elements in $H$, which is a group.
For other basic notions of Hopf algebras, we follow the textbook~\mcite{Mon}.

\section{Basic notions}\label{sec:mp}

First we recall the notion of matched pairs of Hopf algebras
in Majid's book~\cite{Ma}.
\begin{defn}[{\cite[\S~7.2]{Ma}}]\label{defn:MP_Hopf}
A {\bf matched pair of Hopf algebras} is a quadruple $(H,K,\rightharpoonup,\leftharpoonup)$, where $H$ and $K$
are Hopf algebras, $\rightharpoonup:H\otimes K\to K$ is a left $H$-module coalgebra action on $K$, $\leftharpoonup:H\otimes K\to H$ is a right $K$-module coalgebra action on $H$, namely $\rightharpoonup$ and $\leftharpoonup$ are coalgebra maps, such that\vspace{-.5em}
\begin{eqnarray}
\label{eq:MP1}
x\rightharpoonup ab&=&(x_1\rightharpoonup a_1)((x_2\leftharpoonup a_2)\rightharpoonup b),\\
\label{eq:MP2}
x\rightharpoonup 1_K&=&\varepsilon_H(x)1_K,\\
\label{eq:MP3}
xy\leftharpoonup a&=&(x\leftharpoonup(y_1\rightharpoonup a_1))(y_2\leftharpoonup a_2),\\
\label{eq:MP4}
1_H\leftharpoonup a&=&\varepsilon_K(a)1_H,\\
\label{eq:MP5}
(x_1\rightharpoonup a_1) \otimes (x_2\leftharpoonup a_2)  &=&
(x_2\rightharpoonup a_2) \otimes (x_1\leftharpoonup a_1)
\end{eqnarray}
for any $x,y\in H$ and $a,b\in K$.

For a matched pair of Hopf algebras $(H,K,\rightharpoonup,\leftharpoonup)$, the {\bf double cross product} $K\bowtie H$ is a Hopf algebra structure on $K\otimes H$ equipped with the product
\begin{eqnarray}\label{eq:dcp}
(a\otimes x)(b\otimes y) &\coloneqq& a(x_1\rightharpoonup b_1) \otimes (x_2\leftharpoonup b_2)y,\quad\forall x,y\in H,\ a,b\in K,
\end{eqnarray}
and the usual tensor coproduct. The antipode of $K\bowtie H$ is given by
\begin{eqnarray}\label{eq:dcp-anti}
S_{\bowtie}(a\otimes x) &=& (1\otimes S_H(x))(S_K(a)\otimes 1)\ =\ (S_H(x_1)\rightharpoonup S_K(a_1))\otimes(S_H(x_2)\leftharpoonup S_K(a_2)).
\end{eqnarray}
\end{defn}

By \cite[Proposition~21.6]{Maj1}, we have the following characterization of double cross products via factorization.
\begin{prop}\label{prop:MP_Hopf}
With the notations in Definition~\ref{defn:MP_Hopf}, $(H,K,\rightharpoonup,\leftharpoonup)$ is a matched pair of Hopf algebras if and only if there exist a Hopf algebra $A$ and injective Hopf algebra homomorphisms $\iota_K:K\to A$, $\iota_H:H\to A$ such that
the map
$$\xi:K\otimes H\to A,\,a\otimes x\mapsto \iota_K(a)\iota_H(x)$$
is a linear isomorphism. Namely, $A$ is a Hopf algebra factorization into the two Hopf subalgebras $\iota_K(K),\ \iota_H(H)$ and $A\cong K\bowtie H$.
\end{prop}

In the recent work~\cite{FS}, Ferri and Sciandra particularly considered a certain subclass of matched pairs of Hopf algebras, then introduced two equivalent notions, matched pair of actions on a Hopf algebra and Yetter-Drinfeld brace.
\begin{defn}[\cite{FS}]\label{defn:mpa}
A {\bf matched pair of actions} on a Hopf algebra $H$ is a matched pair of Hopf algebras $(H,H,\rightharpoonup,\leftharpoonup)$ satisfying
\begin{eqnarray}
\label{eq:MP*}
xy&=&(x_1\rightharpoonup y_1)(x_2\leftharpoonup y_2),\quad\forall x,y\in H.
\end{eqnarray}
It will be abbreviated as $(H,\rightharpoonup,\leftharpoonup)$.
\end{defn}

Next we recall the notion of Yetter-Drinfeld modules and Hopf algebras in such a category.
\begin{defn}
A (left-left) {\bf Yetter-Drinfeld module} over a Hopf algebra $H$ is a left $H$-module $M$ which is also a left $H$-comodule satisfying
\begin{equation}\label{eq:yd-mod}
\rho(x\cdot m)=x_1m_{-1}S(x_3)\otimes (x_2\cdot m_0),\quad \forall x\in H,\ m\in M,
\end{equation}
where $\rho:M\to H\otimes M$ is the coaction map under the notation $\rho(m)=m_{-1}\otimes m_0$ and $\cdot$ denotes the left $H$-action on $M$.
A morphism of Yetter-Drinfeld modules over $H$ is a left $H$-module and also left $H$-comodule map. Denote by ${_H^H}\YD$ the (pre-)braided tensor category of Yetter-Drinfeld modules over $H$. The tensor product of two Yetter-Drinfeld modules $M,N$ over $H$ has the following action and coaction of $H$ respectively,
$$x\cdot(m\otimes n)=(x_1\cdot m)\otimes (x_2\cdot n),\quad \rho(m\otimes n)=m_{-1}n_{-1}\otimes (m_0\otimes n_0).$$
The (pre-)braiding $c$ of ${_H^H}\YD$ is given by
$$c_{M,N}(m\otimes n)=(m_{-1}\cdot n)\otimes m_0.$$

For a (braided) Hopf algebra $(L,\mu_L,1,\Delta_L,\varepsilon_L,S_L)$ in the braided tensor category ${_H^H}\YD$, it satisfies the compatibility condition $\Delta_L\mu_L=(\mu_L\otimes \mu_L)(\id_L\otimes c_{L,L}\otimes \id_L)(\Delta_L\otimes \Delta_L)$, namely
\begin{equation}\label{eq:yd-hopf}
(ab)_1\otimes (ab)_2=a_1((a_2)_{-1}\cdot b_1)\otimes (a_2)_0b_2,\quad\forall a,b\in L.
\end{equation}
Also, all its structure maps are morphisms in ${_H^H}\YD$.
In particular, $S_L\mu_L=\mu_Lc_{L,L}(S_L\otimes S_L)$, namely
$$S_L(ab)=S_L(a_{-1}\cdot b)S_L(a_0)=(a_{-1}\cdot S_L(b))S_L(a_0),$$
and $\Delta_L S_L=(S_L\otimes S_L)c_{L,L}\Delta_L$, namely
\begin{equation}\label{eq:yd-antipode}
S_L(a)_1\otimes S_L(a)_2=S_L((a_1)_{-1}\cdot a_2)\otimes S_L((a_1)_0)
=((a_1)_{-1}\cdot S_L(a_2)) \otimes S_L((a_1)_0).
\end{equation}
\end{defn}

Also, we need the notion of Doi-Hopf modules~\cite{Doi}, unifying several kinds of modules, e.g. (relative) Yetter-Drinfeld modules and Hopf modules.
\begin{defn}
Given a Hopf algebra $H$, a left $H$-comodule algebra $A$ and a right $H$-module coalgebra $C$, a (right-left) $(H,A,C)$-{\bf Doi-Hopf module} $M$ is a right $A$-module $(M,\cdot)$ and also a left $C$-comodule $(M,\rho)$, such that
\begin{equation}\label{eq:doi-hopf}
\rho(m\cdot a)=(m_{-1}\leftharpoonup a_{-1})\otimes (m_0\cdot a_0),\quad \forall a\in A,\ m\in M,
\end{equation}
where we use $\leftharpoonup$ to denote the right $H$-action on $C$.
A morphism of $(H,A,C)$-Doi-Hopf modules is a right $A$-module and also left $C$-comodule map.
\end{defn}

\section{Double cross products with projections}\label{sec:proj}
In this section, we focus on the structure of a double cross product $K\bowtie H$ with a Hopf algebra projection onto $H$ for later discussion. Based on the biproduct or bosonization theory, we especially obtain a braided Hopf algebra in the category of Yetter-Drinfeld modules over $H$ as a transmutation of $K$.
Also, we illustrate this framework by the typical example from generalized quantum doubles.

\begin{defn}[\cite{Ra,Ma0}]
Given a braided Hopf algebra $B$ in the category ${_H^H}\YD$ of Yetter-Drinfeld modules over a Hopf algebra $H$, there is an ordinary Hopf algebra structure on the tensor product $B\otimes H$, called the {\bf biproduct} or {\bf bosonization} of $B$ by $H$.
It is the smash product Hopf algebra $B\# H$ equipped with the multiplication
\begin{eqnarray}\label{eq:bp_mul}
(a\otimes x)(b\otimes y)&\coloneqq&a(x_1 \cdot b)\otimes x_2y,\quad \forall a,b\in B,\,x,y\in H,
\end{eqnarray}
and the following coproduct
\begin{eqnarray}\label{eq:bp_co}
\Delta(a\otimes x)&\coloneqq&(a^1\otimes (a^2)_{-1}x_1)\otimes ((a^2)_0\otimes x_2),
\end{eqnarray}
where the coproduct of $B$ is written as $\Delta_B(a)=a^1\otimes a^2$. The unit and the counit of $B\# H$ are the trivial ones.
\end{defn}

Conversely, we have the following result by Radford's biproduct theory (\cite[Theorem~3]{Ra}). See also \cite[\S 1.5]{AS}.
\begin{prop}\label{prop:biproduct}
Given Hopf algebras $A$ and $H$, if there are Hopf algebra homomorphisms
$\xymatrix{A \ar@<0.5ex>[r]^{\theta}& H  \ar@<0.5ex>[l]^{\iota}}$
 such that $\theta\iota=\id_H$, then the subalgebra of coinvariants of $\theta$
$$B\coloneqq A^{{\rm co}\,\theta}=\{a\in A\,|\,a_1\otimes \theta(a_2)=a\otimes 1_H\}=\{a_1\iota(S_H(\theta(a_2)))\,|\,a\in A\}$$
is a braided Hopf algebra in ${_H^H}\YD$ with respect to the coproduct
\begin{equation}\label{eq:braid-co}
\Delta_B(a)=a_1\iota(S_H(\theta(a_2)))\otimes a_3,
\end{equation}
the counit $\vep_B=\vep_A\big|_B$, the antipode
\begin{equation}\label{eq:braid-anti}
S_B(a)=\iota(\theta(a_1))S_A(a_2),
\end{equation}
and the action $\cdot$ and the coaction $\rho$ as follows,
\begin{equation}\label{eq:braid-action}
x\cdot a= \iota(x_1)a\iota(S_H(x_2)),\quad \rho(a)=\theta(a_1)\otimes a_2,\quad \forall a\in B,\,x\in H.
\end{equation}
Moreover, there is a Hopf algebra isomorphism
\begin{equation}\label{eq:braid-iso}
\Theta:A\to B\# H,\ a\mapsto a_1\iota(S_H(\theta(a_2)))\otimes \theta(a_3),
\end{equation}
where $B\# H$ is the biproduct of $B$ by $H$.
\end{prop}

\begin{defn}\label{defn:dcp_proj}
Let $K,H$ and $A$ be Hopf algebras, where the underlying coalgebra of $A$ is the tensor coalgebra $K\otimes H$ and the following conditions hold:
\begin{enumerate}[(i)]
\item\label{con:dcp-1}
The natural inclusions $\iota_K:K\to A,\ a\mapsto a\otimes 1$ and $\iota_H:H\to A,\ x\mapsto 1\otimes x$ are Hopf algebra homomorphisms such that
$$\iota_K(a)\iota_H(x)=a\otimes x,\quad\forall a\in K,\,x\in H.$$
\item\label{con:dcp-2}
There is a Hopf algebra homomorphism $\theta:A \to H$ such that $\theta\iota_H=\id_H$.
\end{enumerate}
We call $A$ a {\bf double cross product with a projection $\theta$} onto $H$, denoted by $(K\bowtie H,\theta)$.
\end{defn}

According to Proposition~\ref{prop:MP_Hopf}, condition~\ref{con:dcp-1}, which leads to the double cross product structure $K\bowtie H$, is equivalent to the existence of a matched pair of Hopf algebras $(H,K,\rightharpoonup,\leftharpoonup)$, where
$$x\rightharpoonup a=(\id_K\otimes\varepsilon_H)(\iota_H(x)\iota_K(a)),\quad
x\leftharpoonup a=(\varepsilon_K\otimes\id_H)(\iota_H(x)\iota_K(a)),\quad\forall a\in K,\,x\in H.$$
Since $x\otimes a\mapsto \iota_H(x)\iota_K(a)$ is a coalgebra map from $H\otimes K$ to $K\otimes H$, and
$$(\id_K\otimes\varepsilon_H\otimes \varepsilon_K\otimes\id_H)\Delta_{K\otimes H} (a\otimes x)=a_1\varepsilon_H(x_1)\otimes \varepsilon_K(a_2)x_2=a\otimes x,$$
it particularly gives that $\iota_H(x)\iota_K(a)=(x_1\rightharpoonup a_1) \otimes (x_2\leftharpoonup a_2)$.

Now we further obtain the following crucial equivalence description for double cross products with projections.
\begin{prop}\label{prop:dcp-hopf-map}
For a matched pair of Hopf algebras $(H,K,\rightharpoonup,\leftharpoonup)$, the double cross product $K\bowtie H$ has a projection $\theta$ onto $H$ such that $\theta\iota_H=\id_H$
if and only if there exists a Hopf algebra homomorphism $\pi:K\to H$ such that
\begin{equation}\label{eq:pi-MP*}
\pi(x_1\rightharpoonup a_1)(x_2\leftharpoonup a_2)=x\pi(a),\quad\forall a\in K,\,x\in H.
\end{equation}
\end{prop}
\begin{proof}
First, if there is a Hopf algebra homomorphism $\theta:K\bowtie H \to H$ such that $\theta\iota_H=\id_H$, we denote $\pi\coloneqq\theta\iota_K$. Then, $\pi:K\to H$ is a Hopf algebra homomorphism. Moreover, for any $a\in K$, $x\in H$, we have
\begin{align*}
\pi(x_1\rightharpoonup a_1)(x_2\leftharpoonup a_2)
&\stackrel{\rm{\ref{con:dcp-2}}}{=} \theta(\iota_K(x_1\rightharpoonup a_1))
\theta(\iota_H(x_2\leftharpoonup a_2)) = \theta(\iota_K(x_1\rightharpoonup a_1)\iota_H(x_2\leftharpoonup a_2))\\
&= \theta((x_1\rightharpoonup a_1)\otimes(x_2\leftharpoonup a_2))= \theta(\iota_H(x)\iota_K(a))\\
&= \theta(\iota_H(x))\theta(\iota_K(a))\stackrel{\rm{\ref{con:dcp-2}}}{=}
x\pi(a).
\end{align*}
That is, Eq. ~\eqref{eq:pi-MP*} holds.

Conversely, if there exists a Hopf algebra homomorphism $\pi:K\to H$ satisfying Eq.~\eqref{eq:pi-MP*}, we show that the map
$$\theta:K\bowtie H\to H,\ a\otimes x\mapsto \pi(a)x.$$
is a Hopf algebra projection. Indeed, it is clearly a coalgebra map. On the other hand,
\begin{align*}
\theta((a\otimes x)(b\otimes y))&\stackrel{\eqref{eq:dcp}}{=} \theta(a(x_1\rightharpoonup b_1) \otimes (x_2\leftharpoonup b_2)y)=\pi(a(x_1\rightharpoonup b_1))(x_2\leftharpoonup b_2)y\\
&=\pi(a)\pi(x_1\rightharpoonup b_1)(x_2\leftharpoonup b_2)y\stackrel{\eqref{eq:pi-MP*}}{=} \pi(a)x\pi(b)y
=\theta(a\otimes x)\theta(b\otimes y).
\end{align*}
Hence, $\theta$ is a Hopf algebra homomorphism.
Also, for any $x\in H$, we have
$$\theta(\iota_H(x))=\theta(1_K\otimes x)=\pi(1_K)x=x,$$
so $\theta$ is a Hopf algebra projection such that $\theta\iota_H=\id_H$.
\end{proof}

\begin{remark}\label{rk:dcp}
In \cite{Li}, we has particularly studied the double cross products $H\bowtie H$ with the multiplication $\mu_H$ of $H$ as the projection, namely $K=H$ and $\pi=\id_H$ in Proposition~\ref{prop:dcp-hopf-map}, which also leads to the notion of matched pairs of actions~(Definition~\ref{defn:mpa}). Such an algebraic structure was shown to be equivalent to any one of the following items:
\begin{itemize}
\item
braiding operators on a Hopf algebra~\cite{GGV,GGV0},
\item
Yetter-Drinfeld braces~\cite{FS},
\item
Yetter Drinfeld post-Hopf algebras~\cite{Sc},
\end{itemize}
and can produce set-theoretic type solutions of the braid equation.
\end{remark}

\begin{exam}\label{ex:trivial}
In the context of Proposition~\ref{prop:dcp-hopf-map},
if the right action $\leftharpoonup$ is trivial, then $K\bowtie H$ is the smash product Hopf algebra $K\#_\rightharpoonup H$.
In this situation, Eqs.~\eqref{eq:MP1} and \eqref{eq:MP2} imply that $(K,\rightharpoonup)$ is a left module bialgebra over $H$, and Eq.~\eqref{eq:MP5} becomes
$$(x_1\rightharpoonup a)\otimes x_2=(x_2\rightharpoonup a)\otimes x_1,\quad\forall a\in K,\,x\in H.$$
Moreover, Eq.~\eqref{eq:pi-MP*} has the following equivalent form:
$$\pi(x\rightharpoonup a)=x_1\pi(a)S_H(x_2)=\ad_x\pi(a),\quad\forall a\in K,\,x\in H.$$
Namely, the Hopf algebra homomorphism $\pi:K\to H$ is also an $H$-module homomorphism from $(K,\rightharpoonup)$ to $(H,\ad)$.

Instead, if the left action $\rightharpoonup$ is trivial, then $K\bowtie H = K {_\leftharpoonup}\# H$, and Eqs.~\eqref{eq:MP3}, \eqref{eq:MP4} imply that $(H,\leftharpoonup)$ is a right module bialgebra over $K$, and Eq.~\eqref{eq:MP5} becomes
$$a_1\otimes (x\leftharpoonup a_2)=a_2\otimes (x\leftharpoonup a_1),\quad\forall a\in K,\,x\in H.$$
Moreover, Eq.~\eqref{eq:pi-MP*} means that the right action $\leftharpoonup$ is given by
$$x\leftharpoonup a=S_H(\pi(a_1))x\pi(a_2),\quad\forall a\in K,\,x\in H,$$
which is the pull-back of the right adjoint action of $H$ on itself by $\pi$.
\end{exam}

Now we apply Radford's biproduct theory to double cross products $K\bowtie H$ with  projections onto $H$, in order to transmute the ordinary Hopf algebra $K$ into a braided Hopf algebra in the Yetter-Drinfeld module category ${_H^H}\YD$.
\begin{theorem}\label{thm:dcp-yd-hopf}
Let $H$ and $K$ be Hopf algebras such that $K\bowtie H$ is a double cross product with a projection $\theta$ onto $H$. Let $\pi=\theta\iota_K$. There is a  braided Hopf algebra in ${_H^H}\YD$, namely
$K_{\pi}=(K,\bullet_{\pi},1_K,\Delta_K,\vep_K, S_{\pi})$, where
\begin{eqnarray}\label{eq:ydp-prod}
a \bullet_{\pi} b &=& a_1(\pi(S_K(a_2))\rightharpoonup b),\\
\label{eq:ydp-antipode}
S_{\pi}(a) &=& \pi(a_1)\rightharpoonup S_K(a_2),
\end{eqnarray}
and $H$ acts on $K_{\pi}$ via $\rightharpoonup$, and coacts on it via
$\Ad_\pi:a\mapsto \pi(a_1S_K(a_3))\otimes a_2$.

Moreover, the map $\beta:K\bowtie H\to K_{\pi}\#_\rightharpoonup H,\ a\otimes x\mapsto a_1\otimes \pi(a_2)x$ is a Hopf algebra isomorphism.

\end{theorem}

\begin{proof}
First by Proposition~\ref{prop:dcp-hopf-map}, the map $\pi=\theta\iota_K:K\to H$ is a Hopf algebra homomorphism satisfying Eq.~\eqref{eq:pi-MP*}. According to Proposition~\ref{prop:biproduct}, for $\xymatrix{K\bowtie H \ar@<0.5ex>[r]^{\quad \theta}& H  \ar@<0.5ex>[l]^{\quad \iota_H}}$,
the corresponding subalgebra of coinvariants
$$B\coloneqq (K\bowtie H)^{{\rm co}\,\theta}={\rm span}\{a_1\otimes S_H(\pi(a_2))\,|\,a\in K\}$$
is a braided Hopf algebra in ${_H^H}\YD$. So it is natural to define a linear map
$$\Phi:K\to B,\ a\mapsto a_1\otimes S_H(\pi(a_2)),$$
with the inverse $\Phi^{-1}$ given by $\id_K\otimes\vep_H$.

Since the coproduct of $B$ is computed as
\begin{align*}
\Delta_B(a_1\otimes S_H(\pi(a_2)))&\stackrel{\eqref{eq:braid-co}}{=}
(a_1\otimes S_H(\pi(a_6)))(1_K\otimes S_H(\pi(a_2)S_H(\pi(a_5))))\otimes (a_3\otimes S_H(\pi(a_4)))\\
&=(a_1\otimes S_H(\pi(a_2)S_H(\pi(a_5))\pi(a_6)))\otimes (a_3\otimes S_H(\pi(a_4)))\\
&=(a_1\otimes S_H(\pi(a_2))) \otimes (a_3\otimes S_H(\pi(a_4))),
\end{align*}
$\Phi$ is clearly a coalgebra isomorphism. Therefore, we can transplant the Hopf algebra structure of $B$ in ${_H^H}\YD$ to $K$ via $\Phi$, denoted by $K_\pi$. Namely, we define another multiplication $\bullet_\pi$ on $K$ by
\begin{align*}
a\bullet_\pi b&\coloneqq \Phi^{-1}(\Phi(a)\Phi(b))= (\id_K\otimes\vep_H)((a_1\otimes S_H(\pi(a_2)))(b_1\otimes S_H(\pi(b_2))))\\
&=a_1(S_H(\pi(a_2))\rightharpoonup b)=a_1(\pi(S_K(a_2))\rightharpoonup b),
\end{align*}
and $1_K$ is clearly the unit of $(K,\bullet_\pi)$. Also, we compute that
\begin{align*}
\Phi^{-1}(S_B(\Phi(a)))&\stackrel{\eqref{eq:braid-anti}}{=}
(\id_K\otimes\vep_H)((1_K\otimes\theta(a_1\otimes S_H(\pi(a_4))))
S_{\bowtie}(a_2\otimes S_H(\pi(a_3))))\\
&\stackrel{\eqref{eq:dcp-anti}}{=}
(\id_K\otimes\vep_H)((1_K\otimes\pi(a_1)S_H(\pi(a_6)))\\
&\qquad\cdot((S_H^2(\pi(a_5))\rightharpoonup S_K(a_2))\otimes (S_H^2(\pi(a_4))\leftharpoonup S_K(a_3))))\\
&=\pi(a_1)S_H(\pi(a_4))\rightharpoonup(S_H^2(\pi(a_3))\rightharpoonup S_K(a_2))\\
&=\pi(a_1)S_H(S_H(\pi(a_3))\pi(a_4))\rightharpoonup S_K(a_2)\\
&= \pi(a_1)\rightharpoonup S_K(a_2)= S_\pi(a), \\[.5em]
\Phi^{-1}(x\rightharpoonup\Phi(a))&=
(\id_K\otimes\vep_H)(x\rightharpoonup (a_1\otimes S_H(\pi(a_2))))\\
&\stackrel{\eqref{eq:braid-action}}{=}
(\id_K\otimes\vep_H)((x_1\rightharpoonup a_1) \otimes (x_2\leftharpoonup a_2)S_H(x_3\pi(a_3)))=x\rightharpoonup a,\\[.5em]
(\id_H\otimes\Phi^{-1})\rho(\Phi(a))
&\stackrel{\eqref{eq:braid-action}}{=}(\id_H\otimes\id_K\otimes\vep_H)
(\pi(a_1)S_H(\pi(a_4)) \otimes (a_2\otimes S_H(\pi(a_3))))\\
&=\pi(a_1)S_H(\pi(a_3)) \otimes a_2 = \Ad_\pi(a).
\end{align*}
Hence, we obtain the desired Hopf algebra $K_{\pi}=(K,\bullet_{\pi},1_K,\Delta_K,\vep_K, S_{\pi})$ in ${_H^H}\YD$ with respect to the action $\rightharpoonup$ and the coaction
$\Ad_\pi$. Moreover, Eq.~\eqref{eq:braid-iso} implies a Hopf algebra isomorphism
$$\Theta:K\bowtie H\to B\#_\rightharpoonup H,\ a\otimes x \mapsto (a_1\otimes S_H(\pi(a_2)))\otimes \pi(a_3)x,$$
and $\beta:K\bowtie H\to K_\pi\#_\rightharpoonup H$ is given by $(\Phi^{-1}\otimes\id_H)\Theta$, so $\beta(a\otimes x)=a_1\otimes\pi(a_2)x$.
\end{proof}

A typical construction of matched pairs of Hopf algebras $(H,K,\rightharpoonup,\leftharpoonup)$ is by using (Hopf) skew pairing between $K$ and $H$. When $H$ has the bijective antipode, a bilinear map $\langle\cdot,\cdot\rangle:K\otimes H\to\bk$ is called a {\bf (Hopf) skew pairing}, if
\begin{align*}
  &\langle ab,x \rangle= \langle a,x_1 \rangle\langle b,x_2 \rangle,\quad \langle a,xy \rangle= \langle a_1,y \rangle\langle a_2,x \rangle,\\
  &\langle 1_K,x \rangle= \vep_H(x),\quad \langle a,1_H \rangle= \vep_K(a),\quad \langle S_K(a),x \rangle= \langle a,S_H^{-1}(x) \rangle.
\end{align*}
Then, there exists the following matched pair of Hopf algebras $(H,K,\rightharpoonup,\leftharpoonup)$~(see e.g. \cite{DT}):
\begin{align}\label{eq:sp-action}
x \rightharpoonup a= \langle a_1,x_1\rangle \langle a_3,S_H^{-1}(x_2)\rangle a_2,\quad
x \leftharpoonup a= \langle a_1,x_1\rangle \langle S_K(a_2),x_3\rangle  x_2,\quad \forall a\in K,\,x\in H.
\end{align}
The corresponding double cross product $D(K,H)=K\bowtie H$ is called the {\bf generalized quantum double}, with the famous Drinfeld double $D(H)=H^{* \rm cop}\bowtie H$ as the classical case (via the natural dual pairing between $H^{* \rm cop}$ and $H$).

Applying the general framework of Proposition~\ref{prop:dcp-hopf-map} and Theorem~\ref{thm:dcp-yd-hopf}, we recovers the following result, which can be found e.g. in~\cite[\S 3]{BB},
for generalized quantum doubles with projections.
\begin{prop}\label{prop:gqd}
The generalized quantum double $D(K,H)=K \bowtie H$ has a projection onto $H$ if and only if there exists a Hopf algebra homomorphism $\pi:K\to H$ such that
\begin{equation}\label{eq:sp-MP*}
x\pi(a)=\langle a_1,x_1\rangle\langle a_3,S_H^{-1}(x_3)\rangle \pi(a_2)x_2,\quad\forall a\in K,\,x\in H.
\end{equation}
In this situation, there exists a braided Hopf algebra $K_{\pi}=(K,\bullet_{\pi},1_K,\Delta_K,\vep_K, S_{\pi})$ in ${_H^H}\YD$ with respect to the action $\rightharpoonup$ in Eq.~\eqref{eq:sp-action} and the coaction
$\Ad_\pi$, where
\begin{eqnarray*}
a \bullet_{\pi} b
&=& \langle b_1S_K(b_3),\pi(S_K(a_2))\rangle a_1b_2,\\
S_{\pi}(a) &=& \langle S_K(a_4)S_K^2(a_2),\pi(a_1)\rangle S_K(a_3).
\end{eqnarray*}
\end{prop}

\begin{coro}
In the context of Proposition~\ref{prop:gqd}, if $H$ (resp. $K$) is cocommutative, then the right action $\leftharpoonup$ (resp. the left action $\rightharpoonup$) is trivial, and it specializes to the situation in Example~\ref{ex:trivial}.
\end{coro}

\begin{remark}
When $K$ is cocommutative, the transmutation in Proposition~\ref{prop:gqd} is trivial, namely $K_\pi=K$. By contrast, the transmutation in Theorem~\ref{thm:dcp-yd-hopf} could be highly nontrivial even for cocommutative case, e.g. the universal enveloping algebra of a post Lie algebra and its subadjacent Hopf algebra~(see~\cite{ELM,LST}). Hence, the framework of Theorem~\ref{thm:dcp-yd-hopf} for double cross products with projections is much broader than the situation of generalized quantum doubles.
\end{remark}

\begin{exam}
In \cite[Theorem~1.5]{HN}, Habbestad and Neshveyev generalized Majid's transmutation theory~\cite{Ma1} by showing that a Hopf algebra homomorphism $\pi:K\to H$, in which $(H,r)$ is coquasitriangular, induces a braided Hopf algebra $K_r$ in ${_H^H}\YD$. It is a Tannaka-Krein type reconstruction for braided Hopf algebras via a monoidal functor between the categories of comodules.

Now we show that it can be obtained by Proposition~\ref{prop:gqd}.
Recall that a {\bf coquasitriangular} Hopf algebra $(H,r)$ consists of a Hopf algebra $H$ and a convolution-invertible bilinear map $r:H\otimes H\to\bk$ such that
\begin{eqnarray*}
r(x_1\otimes y_1)x_2y_2&=&r(x_2\otimes y_2)y_1x_1,\\
r(xy\otimes z)&=&r(x\otimes z_1)r(y\otimes z_2),\\
r(x\otimes yz)&=&r(x_1\otimes z)r(x_2\otimes y)
\end{eqnarray*}
for any $x,y,z\in H$. Now given a Hopf algebra homomorphism $\pi:K\to H$, one can define the following skew pairing between $K$ and $H$,
$$\langle a,x\rangle=r^{-1}(x,\pi(a))=r(S_H(x),\pi(a)),\quad \forall a\in K,\,x\in H,$$
such that Eq.~\eqref{eq:sp-MP*} holds. Then, the corresponding  matched pair $(H,K,\rightharpoonup,\leftharpoonup)$ is given by
$$x \rightharpoonup a = r^{-1}(x_1,\pi(a_1))r(x_2,\pi(a_3))a_2, \quad  x \leftharpoonup a = r^{-1}(x_1,\pi(a_1))r(x_3,\pi(a_2))x_2 \quad \forall a\in K,\,x\in H,$$
and the above $K_r$ in \cite{HN} is exactly the braided Hopf algebra $K_\pi$ obtained by Proposition~\ref{prop:gqd}.
\end{exam}

Moreover, as the dual version of \cite[Theorem 4.8, Corollary 4.9]{Zh}, there is the following braided tensor equivalence between the relative Yetter-Drinfeld module category ${_K^H}\YD$ and
the category ${_{K_r}}({^H}\calm)$ consisting of left $K_r$-modules in the category of left $H$-comodules. Now we still wonder how to extend such a result to the more general circumstance in Proposition~\ref{prop:gqd}.
\begin{prop}
Given a coquasitriangular Hopf algebra $(H,r)$ and a Hopf algebra homomorphism $\pi:K\to H$, if $(M,\cdot,\rho)$ is a left $K_r$-module in the category of left $H$-comodules, then
$(M,\rightharpoonup,\rho)$ is a relative Yetter-Drinfeld module in ${_K^H}\YD$ such that
$$a\rightharpoonup m=r(\pi(S_K(a_2)),m_{-1})\,a_1\cdot m_0,\quad\forall a\in K,\,m\in M,$$
and vice versa, since we conversely have
$$a\cdot m=r(\pi(S_K^2(a_2)),m_{-1})\,a_1\rightharpoonup m_0,\quad\forall a\in K,\,m\in M.$$

Correspondingly, ${_{K_r}}({^H}\calm)$ can be endowed with the following braided tensor category structure from ${_K^H}\YD$. For any $M,N\in {_{K_r}}({^H}\calm)$, we have
\begin{align*}
a\cdot(m\otimes n) & = r(\pi(S_K(a_2S_K(a_4))),m_{-1})\,(a_1\cdot m_0)\otimes (a_3\cdot n_0),\\
\rho(m\otimes n) & = m_{-1}n_{-1}\otimes (m_0\otimes n_0),
\end{align*}
where $a\in K$, $m\in M$ and $n\in N$, and the braiding $c$ of ${_{K_r}}({^H}\calm)$ is given by
$$c_{M,N}(m\otimes n)=
r(\pi(S_K(m_{-1})),n_{-1})\,(m_{-2}\cdot n_0)\otimes m_0,\quad\forall m\in M,\,n\in N.$$
\end{prop}

\begin{exam}
Recall that a {\bf quasitriangular} Hopf algebra $H$ has an invertible element $R$ in $H\otimes H$, called $R$-matrix and satisfying
\begin{align}
\label{eq:almost_cocom}&\Delta^{\rm op}(x)=R\Delta(x)R^{-1},\\
\label{eq:QT1}&(\Delta\otimes\id)R=R_{13}R_{23},\\
\label{eq:QT2}&(\id\otimes \Delta)R=R_{13}R_{12}.
\end{align}
where $R_{12}=s_i\otimes t_i\otimes 1$, $R_{23}=1\otimes s_i\otimes t_i$ and $R_{13}=s_i\otimes 1\otimes t_i$ with $R=s_i\otimes t_i$ using the Einstein summation convention.
In particular, $R^{-1}=(S\otimes \id)R=(\id\otimes S^{-1})R$.

Drinfeld~\cite{D1} and Majid~\cite{Ma2} independently showed that a finite-dimensional Hopf algebra $H$ is quasitriangular if and only if the Drinfeld double $D(H)=H^{* \rm cop} \bowtie H$ has a projection onto $H$. A dual version for coquasitriangular Hopf algebras can be found in \cite[Proposition~3.9]{BB}.

\end{exam}

Now a generalized quantum double version is given also by Proposition~\ref{prop:gqd}.
\begin{coro}\label{coro:db-qt}
Given a Hopf algebra $K$ and a quasitriangular Hopf algebra $(H,R)$ with a skew pairing $\langle\cdot,\cdot\rangle$ between them, the generalized quantum double $D(K,H) = K \bowtie H$ has a projection
$$\theta:D(K,H)\to H: a\otimes x\mapsto \langle a,s_i\rangle t_ix$$
such that $\theta\iota_H=\id_H$, where $R=s_i\otimes t_i$.
\end{coro}
\begin{proof}
It is straightforward to check that the map $\pi:K\to H,\,a\mapsto \langle a,s_i\rangle t_i $ is a Hopf algebra homomorphism satisfying Eq.~\eqref{eq:pi-MP*}.
\end{proof}

In the context of Corollary~\ref{coro:db-qt}, let $\pi=\theta\iota_K$, then the braided Hopf algebra $K_\pi$ in ${_H^H}\YD$ is given by Proposition~\ref{prop:gqd} as follows.
\begin{eqnarray*}
a \bullet_{\pi} b
&=& \langle S_K(a_2),s_i\rangle\langle b_1S_K(b_3),t_i\rangle a_1 b_2,
\\[.5em]
S_{\pi}(a) & =&  \langle a_1,s_i\rangle \langle S_K(a_4)S_K^2(a_2),t_i\rangle S_K(a_3),\\[.5em]
\Ad_\pi(a) &=& \pi(a_1S_K(a_3))\otimes a_2 \ =\ t_i \otimes \langle a_1S_K(a_3),s_i\rangle a_2 \ =\ t_i\otimes (s_i\rightharpoonup a).
\end{eqnarray*}

\section{Relative Rota-Baxter operators on Hopf algebras}\label{sec:rrb}

The notion of relative Rota-Baxter operators on (cocommutative) Hopf algebras was originally introduced in \cite{LST}. Later Sciandra broke down the cocommutativity restriction to propose Yetter-Drinfeld relative Rota-Baxter operators~\cite{Sc}. Based on these works, we define a proper version of relative Rota-Baxter operators on Hopf algebras, as an equivalent algebraic structure to double cross products with projections.
\begin{defn}\label{defn:rrb_hopf}
Let $H$ be an ordinary Hopf algebra and let $L$ be a Hopf algebra in the Yetter-Drinfeld module category ${_H^H}\YD$ with respect to an $H$-action $\rightharpoonup$ and the following $H$-coaction
$$\Ad_\pi:L\to H \otimes L,\ a\mapsto \pi(a_1)S_H(\pi(a_3))\otimes a_2,$$
where $\pi:L\to H$ is a coalgebra homomorphism such that
\begin{equation}\label{eq:rrb_hopf}
\pi(a)\pi(b)=\pi(a_1(\pi(a_2)\rightharpoonup b)),\quad \forall a,b\in L.
\end{equation}
Then, the map $\pi$ is called a {\bf (Yetter-Drinfeld) relative Rota-Baxter operator} on $H$ with respect to $(L,\rightharpoonup)$.
\delete{
A {\bf homomorphism} between two relative Rota-Baxter operators
$\pi:L\to H$ and $\pi':L'\to H'$ is a pair of homomorphisms of algebras and coalgebras $(f:H\to H',\,g:L\to L')$ such that
\begin{eqnarray}\label{eq:rrb-homo}
f\pi=\pi'g,\quad g(x\rightharpoonup a)=f(x)\rightharpoonup' g(a) ,\quad \forall\,x\in H,\,a\in L.
\end{eqnarray}}
\end{defn}

\begin{remark}
Comparing with the relative Rota-Baxter operators on cocommutative Hopf algebras defined in \cite[Definition~3.1]{LST}, the setting in Definition~\ref{defn:rrb_hopf} replaces the original cocommutative $H$-module bialgebra $(L,\rightharpoonup)$ by the braided Hopf algebra $(L,\rightharpoonup,\Ad_\pi)$ in ${_H^H}\YD$.
It is a simplified version of the Yetter-Drinfeld relative Rota-Baxter operators in \cite[Definition~4.2]{Sc} by fixing the coaction of $H$ on $L$ as the coadjoint map $\Ad_\pi$ and relaxing condition (4.2) therein, since this condition automatically holds currently by the following lemma.
\end{remark}

\begin{lemma}\label{lem:rrb-action}
Given a relative Rota-Baxter operator $\pi$ on $H$ with respect to $(L,\rightharpoonup)$, Eq.~\eqref{eq:MP5} holds for the pair of maps $(\rightharpoonup,\leftharpoonup)$, where the map $\leftharpoonup:L\otimes H\to L$ is defined by
\begin{equation}\label{eq:rrb_action}
x\leftharpoonup a \coloneqq S_H(\pi(x_1\rightharpoonup a_1))x_2\pi(a_2),\quad \forall a\in L,\,x\in H.
\end{equation}
Correspondingly, $\leftharpoonup$ is a coalgebra homomorphism.
\end{lemma}
\begin{proof}
Since $(L,\rightharpoonup,\Ad_\pi)$ is a braided Hopf algebra in ${_H^H}\YD$, $(L,\rightharpoonup)$ is particularly an $H$-module coalgebra, and thus
\begin{eqnarray*}
\Ad_\pi(x\rightharpoonup a) &=& \pi(x_1\rightharpoonup a_1)S_H(\pi(x_3\rightharpoonup a_3))\otimes (x_2\rightharpoonup a_2)\\
&\stackrel{\eqref{eq:yd-mod}}{=}& x_1\pi(a_1)S_H(\pi(a_3))S_H(x_3)\otimes (x_2\rightharpoonup a_2),
\end{eqnarray*}
and then by Eq.~\eqref{eq:rrb_action} it is equivalent to
$$(x_2\leftharpoonup a_2) \otimes (x_1\rightharpoonup a_1)=(x_1\leftharpoonup a_1) \otimes (x_2\rightharpoonup a_2),\quad \forall a\in L,\,x\in H.
$$
Namely, Eq.~\eqref{eq:MP5} holds for the pair $(\rightharpoonup,\leftharpoonup)$. Then, we have
\begin{eqnarray*}
\Delta_H(x\leftharpoonup a)&\stackrel{\eqref{eq:rrb_action}}{=}& S_H(\pi(x_2\rightharpoonup a_2))x_3\pi(a_3)\otimes S_H(\pi(x_1\rightharpoonup a_1))x_4\pi(a_4)\\
&=&(x_2\leftharpoonup a_2)\otimes S_H(\pi(x_1\rightharpoonup a_1))x_3\pi(a_3)\\
&\stackrel{\eqref{eq:MP5}}{=}& (x_1\leftharpoonup a_1)\otimes S_H(\pi(x_2\rightharpoonup a_2))x_3\pi(a_3)
\ =\ (x_1\leftharpoonup a_1)\otimes(x_2\leftharpoonup a_2).
\end{eqnarray*}
Also, $\varepsilon_H(x\leftharpoonup a)=\varepsilon_H(x)\varepsilon_L(a)$ by Eq.~\eqref{eq:rrb_action}, so $\leftharpoonup$ is a coalgebra homomorphism.
\end{proof}

Next we show that a double cross product $K\bowtie H$ with a projection onto $H$ naturally induces a relative Rota-Baxter operator on $H$.
\begin{theorem}\label{thm:dcp-rrb}
Given a double cross product with a projection $(K\bowtie H,\theta)$, the map $\pi=\theta\iota_K$ is a relative Rota-Baxter operator on $H$ with respect to $(K_\pi,\rightharpoonup)$, where $K_\pi$ is the braided Hopf algebra in the Yetter-Drinfeld module category ${_H^H}\YD$ defined in Theorem~\ref{thm:dcp-yd-hopf}.
\end{theorem}
\begin{proof}
According to Theorem~\ref{thm:dcp-yd-hopf}, $K_\pi$ is a Hopf algebra in ${_H^H}\YD$ with respect to the action $\rightharpoonup$ and the coaction $\Ad_\pi$. Since $K_\pi$ has the same coalgebra structure as $K$, $\pi:K_{\pi}\to H$ is a coalgebra homomorphism.
Next for any $a,b\in K$, we have
\begin{eqnarray*}
\pi(a_1\bullet_{\pi} (\pi(a_2)\rightharpoonup b))
&\stackrel{\eqref{eq:ydp-prod}}{=}&
\pi(a_1(\pi (S_K(a_2))\rightharpoonup (\pi(a_3)\rightharpoonup b)))\\
&=& \pi(a_1(S_H(\pi (a_2))\pi(a_3)\rightharpoonup b))
\ =\ \pi(ab)= \pi(a)\pi(b).
\end{eqnarray*}
Namely, Eq.~\eqref{eq:rrb_hopf} holds for $\pi:K_{\pi}\to H$.
Therefore, the map $\pi$ is a relative Rota-Baxter operator on $H$ with respect to $(K_\pi,\rightharpoonup)$.
\end{proof}

Conversely, we show that relative Rota-Baxter operators on Hopf algebras can provide double cross products with projections. Therefore, they are equivalent algebraic structures.
\begin{theorem}\label{thm:rrb-dcp}
Given a relative Rota-Baxter operator $\pi$ on $H$ with respect to $(L,\rightharpoonup)$, we define
\begin{eqnarray}\label{eq:rrb-prod}
a \circ_{\pi} b &=& a_1(\pi (a_2)\rightharpoonup b),\\
\label{eq:rrb-antipode}
T_{\pi}(a) &=& S_H(\pi(a_1))\rightharpoonup S_L(a_2)
\end{eqnarray}
for any $a,b\in L$. Denote $L_{[\pi]}=(L,\circ_{\pi},1_L,\Delta_L,\varepsilon_L,T_{\pi})$.
Then $L_{[\pi]}$ is an ordinary Hopf algebra and $(H,L_{[\pi]},\rightharpoonup,\leftharpoonup)$ is a matched pair of Hopf algebras such that $L_{[\pi]}\bowtie H$ is a double cross product with a projection $\theta=\mu_H(\pi\otimes id_H)$ onto $H$, where the map $\leftharpoonup$ is defined by Eq.~\eqref{eq:rrb_action}.
\end{theorem}
\begin{proof}
First we check that $L_{[\pi]}$ is a Hopf algebra for completeness, so $\pi:L_{[\pi]}\to H$ is a Hopf algebra homomorphism by Eqs.~\eqref{eq:rrb_hopf} and \eqref{eq:rrb-prod}. Since $(L,\rightharpoonup,\Ad_\pi)$ is a braided Hopf algebra in ${_H^H}\YD$, Eq.~\eqref{eq:yd-hopf} is written as
$$\Delta_L(ab)=a_1(\pi(a_2)S_H(\pi(a_4))\rightharpoonup b_1)\otimes a_3b_2,\quad \forall a,b\in L.$$
Then, the product $\circ_\pi$ is a coalgebra homomorphism, since
\begin{align*}
\Delta_L(a\circ_{\pi} b)&=
\Delta_L(a_1(\pi (a_2)\rightharpoonup b))=a_1(\pi(a_2)S_H(\pi(a_4))\rightharpoonup (\pi(a_5)\rightharpoonup b_1))\otimes a_3(\pi(a_6)\rightharpoonup b_2)\\
&=a_1(\pi(a_2)\rightharpoonup b_1)\otimes a_3(\pi(a_4)\rightharpoonup b_2)
= (a_1\circ_{\pi} b_1)\otimes (a_2\circ_{\pi} b_2),\\
\varepsilon_L(a\circ_{\pi} b)&=
\varepsilon_L(a_1(\pi (a_2)\rightharpoonup b))=\varepsilon_L(a_1)\varepsilon_H(\pi (a_2))\varepsilon_L(b)=\varepsilon_L(a)\varepsilon_L(b).
\end{align*}
Correspondingly, we have
\begin{align*}
(a \circ_{\pi} b)\circ_{\pi} c&= (a_1 \circ_{\pi} b_1)(\pi(a_2 \circ_{\pi} b_2)\rightharpoonup c)
\stackrel{\eqref{eq:rrb_hopf}}{=}
a_1(\pi (a_2)\rightharpoonup b_1)(\pi(a_3)\pi(b_2)\rightharpoonup c)\\
&=a_1(\pi (a_2)\rightharpoonup b_1)(\pi(a_3)\rightharpoonup(\pi(b_2)\rightharpoonup c))= a_1(\pi(a_2)\rightharpoonup b_1(\pi(b_2)\rightharpoonup c))\\
&= a \circ_{\pi} (b\circ_{\pi} c).
\end{align*}
As $\pi(1_L)$ is a group-like element in $H$, Eq.~\eqref{eq:rrb_hopf} implies that $\pi(1_L)=1_H$, and so we have seen that $(L,\circ_\pi,1_L,\Delta_L,\varepsilon_L)$ is a bialgebra. Moreover,
\begin{align*}
a_1\circ_\pi T_{\pi}(a_2) &= a_1(\pi(a_2)\rightharpoonup (S_H(\pi(a_3))\rightharpoonup S_L(a_4)))\\
&= a_1(\pi(a_2)S_H(\pi(a_3))\rightharpoonup S_L(a_4))
= a_1S_L(a_2) =\varepsilon_L(a)1_L,
\end{align*}
and then $\pi(a_1)\pi(T_\pi(a_2))\stackrel{\eqref{eq:rrb_hopf}}{=}
\pi(a_1\circ_\pi T_\pi(a_2))=\pi(\varepsilon_L(a)1_L)=\varepsilon_H(\pi(a))1_H$,
or equivalently
$$\pi(T_\pi(a))=S_H(\pi(a)).$$
On the other hand,
\begin{align*}
\Delta_L(T_\pi(a))&=\Delta_L(S_H(\pi(a_1))\rightharpoonup S_L(a_2))
=S_H(\pi(a_1))\rightharpoonup \Delta_L(S_L(a_2))\\
&\stackrel{\eqref{eq:yd-antipode}}{=}
(S_H(\pi(a_2))\rightharpoonup
(\pi(a_3)S_H(\pi(a_5))\rightharpoonup S_L(a_6))) \otimes (S_H(\pi(a_1))\rightharpoonup S_L(a_4))\\
&=(S_H(\pi(a_2))\pi(a_3)S_H(\pi(a_5))\rightharpoonup S_L(a_6)) \otimes (S_H(\pi(a_1))\rightharpoonup S_L(a_4))\\
&=(S_H(\pi(a_3))\rightharpoonup S_L(a_4)) \otimes (S_H(\pi(a_1))\rightharpoonup S_L(a_2))\\
&= T_\pi(a_2)\otimes T_\pi(a_1),
\end{align*}
where the second equality holds as $(L,\rightharpoonup)$ is a left $H$-module coalgebra. Therefore, we have
\begin{align*}
T_{\pi}(a_1)\circ_\pi a_2 &= T_{\pi}(a_2)(\pi(T_{\pi}(a_1))\rightharpoonup a_3)= T_{\pi}(a_2)(S_H(\pi(a_1))\rightharpoonup a_3)\\
&= (S_H(\pi(a_2))\rightharpoonup S_L(a_3))(S_H(\pi(a_1))\rightharpoonup a_4)\\
&= S_H(\pi(a_1))\rightharpoonup S_L(a_2)a_3 = \varepsilon_L(a)1_L,
\end{align*}
where the last two equalities hold as $(L,\rightharpoonup)$ is a left $H$-module algebra.

Next we check that $(H,L_{[\pi]},\rightharpoonup,\leftharpoonup)$ is a matched pair of Hopf algebras. By Lemma~\ref{lem:rrb-action}, $\leftharpoonup$ is a coalgebra homomorphism, and Eq.~\eqref{eq:MP5} holds for the pair of maps $(\rightharpoonup,\leftharpoonup)$. Also,
\begin{align*}
x\rightharpoonup (a\circ_\pi b) &= x\rightharpoonup a_1(\pi (a_2)\rightharpoonup b)\\
&=(x_1\rightharpoonup a_1)(x_2\rightharpoonup(\pi (a_2)\rightharpoonup b))\\
&=(x_1\rightharpoonup a_1)(x_2\pi (a_2)\rightharpoonup b)\\
&\stackrel{\eqref{eq:rrb_action}}{=} (x_1\rightharpoonup a_1)(\pi(x_2 \rightharpoonup a_2)(x_3\leftharpoonup a_3)\rightharpoonup b)\\
&=(x_1\rightharpoonup a_1)(\pi(x_2 \rightharpoonup a_2)
\rightharpoonup((x_3\leftharpoonup a_3)\rightharpoonup b))\\
&=(x_1\rightharpoonup a_1)\circ_\pi((x_2\leftharpoonup a_2)\rightharpoonup b).
\end{align*}
Hence, we have
\begin{eqnarray*}
(x\leftharpoonup a)\leftharpoonup b
&\stackrel{\eqref{eq:rrb_action}}{=}&S_H(\pi((x_1\leftharpoonup a_1)\rightharpoonup b_1))(x_2\leftharpoonup a_2)\pi(b_2)\\
&\stackrel{\eqref{eq:rrb_action}}{=}&S_H(\pi((x_1\leftharpoonup a_1)\rightharpoonup b_1))S_H(\pi(x_2\rightharpoonup a_2))x_3\pi(a_3)\pi(b_2)\\
&\stackrel{\eqref{eq:rrb_hopf}}{=}&
S_H(\pi(x_2\rightharpoonup a_2)\pi((x_1\leftharpoonup a_1)\rightharpoonup b_1))x_3
\pi(a_3\circ_\pi b_2)\\
&\stackrel{\eqref{eq:MP5},\,\eqref{eq:rrb_hopf}}{=}&
S_H(\pi((x_1\rightharpoonup a_1)\circ_\pi((x_2\leftharpoonup a_2)\rightharpoonup b_1)))x_3\pi(a_3\circ_\pi b_2)\\
&=& S_H(\pi(x_1\rightharpoonup (a_1\circ_\pi b_1)))x_2\pi(a_2\circ_\pi b_2)
\ \stackrel{\eqref{eq:rrb_action}}{=}\ x\leftharpoonup (a\circ_\pi b),\\[.5em]
x\leftharpoonup 1_L &\stackrel{\eqref{eq:rrb_action}}{=}&
S_H(\pi(x_1\rightharpoonup 1_L))x_2\pi(1_L)\ =\ \varepsilon_H(x_1)x_2\ =\ x.
\end{eqnarray*}
Therefore, $(H,\leftharpoonup)$ is a right module coalgebra over $L_{[\pi]}$. Also,
\begin{eqnarray*}
(x\leftharpoonup (y_1\rightharpoonup a_1))(y_2\leftharpoonup a_2)
&\stackrel{\eqref{eq:rrb_action}}{=}&
S_H(\pi(x_1\rightharpoonup(y_1\rightharpoonup a_1)))x_2 \pi(y_2\rightharpoonup a_2)(y_3\leftharpoonup a_3)\\
&\stackrel{\eqref{eq:rrb_action}}{=}&
S_H(\pi(x_1y_1\rightharpoonup a_1))x_2y_2\pi(a_2)
\ \stackrel{\eqref{eq:rrb_action}}{=}\
xy \leftharpoonup a,\\[.5em]
1_H \leftharpoonup a &\stackrel{\eqref{eq:rrb_action}}{=}& S_H(\pi(1_H\rightharpoonup a_1))\pi(a_2)\ =\ S_H(\pi(a_1))\pi(a_2)\\
  &=& \varepsilon_H(\pi(a))1_H \ =\ \varepsilon_L(a)1_H.
\end{eqnarray*}
Hence, we have checked that Eqs.~\eqref{eq:MP1}--\eqref{eq:MP4} hold. Moreover, Eq.~\eqref{eq:rrb_action} is clearly equivalent to Eq.~\eqref{eq:pi-MP*}, so $L_{[\pi]}\bowtie H$ is a double cross product with a projection $\theta=\mu_H(\pi\otimes id_H)$ onto $H$ by Proposition~\ref{prop:dcp-hopf-map}.
\end{proof}

When $H$ is finite-dimensional, the Drinfeld double $D(H)=H^{* \rm cop}\bowtie H$ is a quasitriangular Hopf algebra with the $R$-matrix
$$\calr=(\vep\otimes h_i) \otimes (h^i\otimes 1),$$
given by any linear basis $\{h_i\}$ of $H$ and its dual basis $\{h^i\}$ in $H^*$. If $D(H)$ has a projection $\theta$ onto $H$, then $H$ is clearly quasitriangular with the $R$-matrix
$$R=(\theta\otimes \theta)(\calr)=h_i\otimes \pi(h^i),$$
where $\pi=\theta\iota_{H^{* \rm cop}}$.
Since $\Hom(H^*,H)\cong H\otimes H$,
we actually have the following result by Corollary~\ref{coro:db-qt}, Theorem~\ref{thm:dcp-rrb} and Theorem~\ref{thm:rrb-dcp}.
\begin{coro}
Given a finite-dimensional Hopf algebra $H$, the Hopf algebra homomorphisms $\pi:H^{* \rm cop}\to H$ satisfying Eq.~\eqref{eq:sp-MP*}, or equivalently the relative Rota-Baxter operators $\pi$ on $H$ with respect to $(H^{* \rm cop}_\pi,\rightharpoonup)$, are in one-to-one correspondence with the quasitriangular structures on $H$, where the action $\rightharpoonup$ is defined as in Eq.~\eqref{eq:sp-action} via the dual pairing between $H^{* \rm cop}$ and $H$.
\end{coro}

From this viewpoint, relative Rota-Baxter operators on Hopf algebras may be regarded as a generalization and enrichment of their quasitriangular structures.

\section{Double cross products with projections and matched pairs of actions}\label{sec:mpa}
As mentioned in Remark~\ref{rk:dcp}, a double cross product $H\bowtie H$ with the multiplication $\mu_H$ as the projection equivalently provides a matched pair of actions on a single Hopf algebra $H$. In this section, we generally study the condition when a double cross product with a projection induces a matched pair of actions.

Let $(H,K,\rightharpoonup,\leftharpoonup)$ be a matched pair of Hopf algebras
and let $\pi:K\to H$ be a Hopf algebra homomorphism satisfying Eq.~\eqref{eq:pi-MP*}  as in Proposition~\ref{prop:dcp-hopf-map}. Define two binary products $\rightharpoonup_\pi$ and $\leftharpoonup_\pi$ on $K$ as follows:
\begin{align}
\label{eq:pi-l-action}
&a\rightharpoonup_\pi b\coloneqq \pi(a) \rightharpoonup b,\\
\label{eq:pi-r-action}
&a\leftharpoonup_\pi b\coloneqq S_K(\pi(a_1) \rightharpoonup b_1)a_2b_2
\end{align}
for any $a,b\in K$.

\begin{theorem}\label{thm:dcp-mpa}
In the above setting, if the map $\leftharpoonup_\pi$ is a coalgebra homomorphism,
then the triple
$(K,\rightharpoonup_\pi,\leftharpoonup_\pi)$ is a matched pair of actions on $K$.
\end{theorem}

\begin{proof}
By definition, $\rightharpoonup_\pi$ is a left $K$-module coalgebra action on itself, and Eq.~\eqref{eq:MP*} holds for the pair of maps $(\rightharpoonup_\pi,\leftharpoonup_\pi)$.
\delete{
Write $S=S_K$ for short.
\begin{eqnarray*}
a\rightharpoonup_\pi b\bullet_\pi c
&\stackrel{\eqref{eq:ydp-prod}}{=}& \pi(a)\rightharpoonup b_1(\pi S(b_2)\rightharpoonup c)\\
&=& (\pi(a_1)\rightharpoonup b_1)((\pi(a_2)\leftharpoonup b_2)\rightharpoonup(\pi S(b_3)\rightharpoonup c))\\
&=& (\pi(a_1)\rightharpoonup b_1)((\pi(a_2)\leftharpoonup b_2)\pi S(b_3)\rightharpoonup c)\\
&\stackrel{\eqref{eq:pi-MP*}}{=}& (\pi(a_1)\rightharpoonup b_1)(S_H\pi(\pi(a_2)\rightharpoonup b_2)\pi(a_3)\pi(b_3)\pi S(b_4)\rightharpoonup c)\\
&=& (\pi(a_1)\rightharpoonup b_1)(\pi S(\pi(a_2)\rightharpoonup b_2)\rightharpoonup(\pi(a_3)\rightharpoonup c))\\
&=& (a_1\rightharpoonup_\pi b)\bullet_\pi(a_2\rightharpoonup_\pi c)\\[.5em]
a\rightharpoonup_\pi (b\rightharpoonup_\pi c)
&=& \pi(ab) \rightharpoonup c\\
&\stackrel{\eqref{eq:ydp-prod}}{=}& \pi(a_1 \bullet_{\pi} (\pi(a_2)\rightharpoonup  b)) \rightharpoonup c\\
&=& a_1 \bullet_{\pi}(a_2\rightharpoonup_\pi b)\rightharpoonup_\pi c
\end{eqnarray*}}

Next we check that $\leftharpoonup_\pi$ is a right $K$-module coalgebra action on itself. By the assumption, we only need to check that
it is a right $K$-module action as follows.
\begin{eqnarray*}
(a\leftharpoonup_\pi b)\leftharpoonup_\pi c
&\stackrel{\eqref{eq:pi-r-action}}{=}& S_K\big(S_K(a_2 \rightharpoonup_\pi b_2)a_3b_3\rightharpoonup_\pi c_1\big)S_K(a_1 \rightharpoonup_\pi b_1)a_4b_4c_2\\
&=& S_K\big((a_1 \rightharpoonup_\pi b_1)\big(S_K(a_2 \rightharpoonup_\pi b_2)a_3b_3\rightharpoonup_\pi c_1\big)\big)a_4b_4c_2\\
&\stackrel{\eqref{eq:pi-l-action}}{=}&
S_K\big((\pi(a_1) \rightharpoonup b_1)\big(S_H(\pi(\pi(a_2) \rightharpoonup b_2))\pi(a_3)\pi(b_3)\rightharpoonup c_1\big)\big)a_4b_4c_2\\
&\stackrel{\eqref{eq:pi-MP*}}{=}&
S_K\big((\pi(a_1) \rightharpoonup b_1)\big((\pi(a_2) \leftharpoonup b_2)\rightharpoonup c_1\big)\big)a_3b_3c_2\\
&\stackrel{\eqref{eq:MP1}}{=}&
S_K(\pi(a_1) \rightharpoonup b_1c_1)a_2(b_2c_2) \ \stackrel{\eqref{eq:pi-r-action}}{=} \ a\leftharpoonup_\pi bc,\\[.5em]
a\leftharpoonup_\pi 1 &\stackrel{\eqref{eq:pi-r-action}}{=}& S_K(a_1 \rightharpoonup_\pi 1)a_2 \ \stackrel{\eqref{eq:MP2}}{=}\ \varepsilon_H(\pi(a_1))a_2 \ =\ \varepsilon_K(a_1)a_2 \ =\ a.
\end{eqnarray*}

Consequently, according to \cite[Theorem~4.1]{Li0}, we obtain that $(K,\rightharpoonup_\pi,\leftharpoonup_\pi)$ is a matched pair of actions on $K$.
\end{proof}

\begin{remark}
The result in \cite[Lemma~3.6]{FS} tells us that the condition of
$\leftharpoonup_\pi$ being a coalgebra homomorphism in Theorem~\ref{thm:dcp-mpa} is equivalent to condition \eqref{eq:MP5} for the pair of maps $(\rightharpoonup_\pi,\leftharpoonup_\pi)$. In this situation, the pair $(K_\pi,\rightharpoonup_\pi)$ is a Yetter-Drinfeld post-Hopf algebra~(\cite[Definition~3.1]{Sc}), where $K_\pi$ is constructed in Theorem~\ref{thm:dcp-yd-hopf}. Moreover, the two functors $((L,\rightharpoonup)\, \mapsto\, \id_L:L\to L_\rightharpoonup)$ and $(\pi:L\to H\,\mapsto\,  (L,\rightharpoonup_\pi))$ provide an adjunction between relative Rota-Baxter operators $\pi$ with coalgebra homomorphisms $\leftharpoonup_\pi$  and  Yetter-Drinfeld post-Hopf algebras.
\end{remark}

\begin{coro}\label{coro:cocomm}
If $K$ is cocommutative, then $(K,\rightharpoonup_\pi,\leftharpoonup_\pi)$ is a matched pair of actions on $K$.
\end{coro}
\begin{proof}
The right action $\leftharpoonup_\pi$ is clearly a coalgebra homomorphism when $K$ is cocommutative, so $(K,\rightharpoonup_\pi,\leftharpoonup_\pi)$ is a matched pair of actions on $K$ by Theorem~\ref{thm:dcp-mpa}.
\end{proof}

\begin{coro}\label{coro:injective}
If $\pi$ is injective, then $(K,\rightharpoonup_\pi,\leftharpoonup_\pi)$ is a matched pair of actions on $K$.
\end{coro}
\begin{proof}
For any $a,b\in K$, we have
$$\varepsilon_K(a\leftharpoonup_\pi b) \stackrel{\eqref{eq:pi-r-action}}{=} \varepsilon_K(S_K(a_1 \rightharpoonup_\pi b_1)a_2b_2) = \varepsilon_K(a \rightharpoonup_\pi b) = \varepsilon_K(a)\varepsilon_K(b).$$
Moreover, since
$$\pi(a)\leftharpoonup b \stackrel{\eqref{eq:pi-MP*}}{=} S_H(\pi(\pi(a_1)\rightharpoonup b_1))\pi(a_2)\pi(b_2)
=\pi(S_K(\pi(a_1)\rightharpoonup b_1)a_2b_2)\stackrel{\eqref{eq:pi-r-action}}{=} \pi(a\leftharpoonup_\pi b),$$
we have
\begin{eqnarray*}
(\pi\otimes\pi)(\Delta_K(a\leftharpoonup_\pi b))&=&
\Delta_H(\pi(a\leftharpoonup_\pi b))\ =\ \Delta_H(\pi(a)\leftharpoonup b)\\
&=&
(\pi(a_1)\leftharpoonup b_1)\otimes(\pi(a_2)\leftharpoonup b_2)
\ =\ (\pi\otimes\pi)((a_1\leftharpoonup_\pi b_1)\otimes(a_2\leftharpoonup_\pi b_2)).
\end{eqnarray*}
So the injectivity of $\pi$ implies that $\leftharpoonup_\pi$ is a coalgebra homomorphism. According to Theorem~\ref{thm:dcp-mpa}, $(K,\rightharpoonup_\pi,\leftharpoonup_\pi)$ is a matched pair of actions on $K$.
\end{proof}

\begin{remark}
In the context of Corollary~\ref{coro:injective}, we actually obtain injective relative Rota-Baxter operators on Hopf algebras by Theorem~\ref{thm:dcp-rrb}, and the adjunction between them and Yetter-Drinfeld post-Hopf algebras has been given in~\cite[Proposition~4.20]{Sc}.
\end{remark}

Given a matched pair of Hopf algebras $(H,K,\rightharpoonup,\leftharpoonup)$, we clearly have the left $K$-comodule algebra $K$ via its coproduct $\Delta_K$, and the right $K$-module coalgebra $H$ with respect to the right action $\leftharpoonup$ of $K$ on $H$.
\begin{theorem}\label{thm:doi-hopf}
Given a matched pair of Hopf algebras $(H,K,\rightharpoonup,\leftharpoonup)$ and a Hopf algebra homomorphism $\pi:K\to H$ satisfying Eq.~\eqref{eq:pi-MP*},
the map $\pi:K\to H$ is a morphism
of $(K,K,H)$-Doi-Hopf modules, where
$H$ is a $(K,K,H)$-Doi-Hopf module with respect to the right $K$-action $\leftharpoonup$ and the left $H$-coaction $\Delta_H$, and
$K$ is a $(K,K,H)$-Doi-Hopf module with respect to the right $K$-action $\leftharpoonup_\pi$ and the left $H$-coaction $(\pi\otimes\id_K)\Delta_K$.
\end{theorem}
\begin{proof}
First $(H,\leftharpoonup,\Delta_H)$ is a $(K,K,H)$-Doi-Hopf module, since Eq.~\eqref{eq:doi-hopf} is written as
$$\Delta_H(x\leftharpoonup a)=(x_1\leftharpoonup a_1)\otimes (x_2\leftharpoonup a_2),\quad \forall x\in H,\ a\in K,$$
which holds for the coalgebra map $\leftharpoonup:H\otimes K\to H$.

On the other hand, we have shown previously that $\leftharpoonup_\pi$ is a right module action of $K$ on itself, and $\pi:(K,\leftharpoonup_\pi)\to (H,\leftharpoonup)$ is a right $K$-module map, that is, $\pi(a\leftharpoonup_\pi b)=\pi(a)\leftharpoonup b$.
So Eq.~\eqref{eq:doi-hopf} also holds for $(K,\leftharpoonup_\pi,(\pi\otimes\id_K)\Delta_K)$ as follows.
\begin{eqnarray*}
(\pi\otimes\id_K)\Delta_K(a\leftharpoonup_\pi b) &\stackrel{\eqref{eq:pi-r-action}}{=}& \pi(S_K(\pi(a_2) \rightharpoonup b_2)a_3b_3) \otimes (S_K(\pi(a_1) \rightharpoonup b_1)a_4b_4)\\
&=& \pi(a_2\leftharpoonup_\pi b_2) \otimes (S_K(\pi(a_1) \rightharpoonup b_1)a_3b_3)\\
&=& (\pi(a_2) \leftharpoonup b_2) \otimes (S_K(\pi(a_1) \rightharpoonup b_1)a_3b_3)\\
&\stackrel{\eqref{eq:MP5}}{=}&
(\pi(a_1) \leftharpoonup b_1) \otimes (S_K(\pi(a_2) \rightharpoonup b_2)a_3b_3)\\
&\stackrel{\eqref{eq:pi-r-action}}{=}& (\pi(a_1) \leftharpoonup b_1) \otimes (a_2\leftharpoonup_\pi b_2).
\end{eqnarray*}
Namely, $(K,\leftharpoonup_\pi,(\pi\otimes\id_K)\Delta_K)$ is a $(K,K,H)$-Doi-Hopf module. Also, we have
$$(\id_K\otimes\pi)((\pi\otimes\id_K)\Delta_K(a))=(\pi\otimes\pi)\Delta_K(a)=\Delta_H(\pi(a)).$$
Hence, the map $\pi:K\to H$ is a morphism of $(K,K,H)$-Doi-Hopf modules.
\end{proof}

According to Theorem~\ref{thm:rrb-dcp} and Theorem~\ref{thm:doi-hopf}, we
have the following reformulation.
\begin{coro}\label{coro:doi-hopf}
Given a relative Rota-Baxter operator $\pi$ on $H$ with respect to $(L,\rightharpoonup)$, the map $\pi:L\to H$ is a morphism
of right-left $(L_{[\pi]},L_{[\pi]},H)$-Doi-Hopf modules.
\end{coro}

In the end, we give the following example from matched pairs of Lie algebras.
\begin{exam}
Let $(\g,[\cdot,\cdot]_\g)$ and $(\h,[\cdot,\cdot]_\h)$ be two Lie algebras such that $(\g,\rhd)$ is a representation of $\h$ and $(\h,\btl)$ is a representation of $\g$. According to~\cite[Theorem~4.1]{Ma3}, the direct-sum space $\g\oplus \h$ has the following Lie bracket
$$[a+x,b+y]=[a,b]_\g+x\rhd b-y\rhd a + a\btl y - b\btl x +[x,y]_\h,\quad\forall a,b\in\g,\ x,y\in\h,$$
if and only if
\begin{align*}
&x\rhd[a,b]_\g = [x\rhd a,b]_\g+[a,x\rhd b]_\g + (b\btl x)\rhd a - (a\btl x)\rhd b,\\
&a\btl[x,y]_\h = [a\btl x,y]_\h+[x,a\btl y]_\h + (y\rhd a)\btl x - (x\rhd a)\btl y
\end{align*}
for any $a,b\in\g$ and $x,y\in\h$.
In this situation, the tuple $(\h,\g,\rhd,\btl)$ is called a {\bf matched pair of Lie algebras}, and the above Lie algebra structure on $\g\oplus \h$ is called the double cross product Lie algebra and denoted by $\g\bowtie \h$.

If there exists a Lie algebra homomorphism $\pi:\g\to \h$ such that
\begin{equation}\label{eq:btl}
a \btl x = [\pi(a),x]_\h+\pi(x\rhd a),\quad\forall a\in\g,\ x\in\h,
\end{equation}
then $\theta:\g\bowtie \h\to\h,\ a+x\to \pi(a)+x$ is a Lie algebra projection. Also, we have the induced Lie algebra $\g_\pi=(\g,[\cdot,\cdot]_\pi)$ defined by
$$[a,b]_\pi\coloneqq [a,b]_\g-\pi(a)\rhd b+\pi(b)\rhd a,\quad\forall a,b\in\g$$
and the map $\rhd$ as a left $\h$-module action on $\g_\pi$, so that $\pi:\g_\pi\to\h$ is a relative Rota-Baxter operator of weight 1 on $\h$ with respect to the action $(\g_\pi,\rhd)$.

Define the right $\g$-module action $\lhd$ on $\h$ by
$$x\lhd a \coloneqq  - a \btl x,\quad\forall a\in\g,\ x\in\h.$$
By~\cite[Example~4.5]{Ma3}, a matched pair of Lie algebras $(\g,\h,\rhd,\btl)$ can extend to a matched pair of Hopf algebras
$(U(\h),U(\g),\rhd,\lhd)$, such that $U(\h)\bowtie U(\g)\cong U(\g\bowtie \h)$.

\end{exam}

In \cite[Theorem~3.6]{LST}, it was shown that relative Rota-Baxter operators on Lie algebras extend to their universal envelopes. In parallel with such a result, we have the following extension.
\begin{prop}
Given a matched pair of Lie algebras $(\h,\g,\rhd,\btl)$, a Lie algebra homomorphism $\pi:\g\to\h$ satisfying Eq.~\eqref{eq:btl} extends to a Hopf algebra homomorphism $\Pi:U(\g)\to U(\h)$  satisfying
Eq.~\eqref{eq:pi-MP*}, which accordingly is a morphism
of $(U(\g),U(\g),U(\h))$-Doi-Hopf modules as in Theorem~\ref{thm:doi-hopf}.
\end{prop}
\begin{proof}
By the universal property of $\U(\g)$, we clearly have the unique Hopf algebra homomorphism $\Pi:U(\g)\to U(\h)$ such that $\Pi|_\g=\pi$. It remains to check that $\Pi$ satisfies Eq.~\eqref{eq:pi-MP*}, which can be done by induction on the length of elements in the universal envelopes.

First by definition, Eq.~\eqref{eq:btl} is rewritten as
$$x\lhd a= [x,\pi(a)]_\h-\pi(x\rhd a),\quad\forall a\in\g,\ x\in\h.$$
Namely, we have the following identity in $U(\h)$,
$$\pi(x\rhd a)+x\lhd a +\pi(a)x=x\pi(a),\quad\forall a\in\g,\ x\in\h.\eqno{(*)}$$
so Eq.~\eqref{eq:pi-MP*} holds for all elements of length $\leq 1$ in the universal envelopes. For any $A\in U(\g)$,
\begin{align*}
&\,\,\quad\Pi(x\rhd aA)+\Pi(aA_1)(x\lhd A_2)+\Pi(A_1)(x\lhd aA_2)\\
&\stackrel{\eqref{eq:MP1}}{=} \pi(x\rhd a)\Pi(A)+\pi(a)\Pi(x\rhd A) +\Pi((x\lhd a)\rhd A)\\
&\qquad +\pi(a)\Pi(A_1)(x\lhd A_2)+\Pi(A_1)((x\lhd a)\lhd A_2)\\
&= \pi(x\rhd a)\Pi(A)+\pi(a)x\Pi(A) +(x\lhd a)\Pi(A)\\ &\stackrel{(*)}{=}x\pi(a)\Pi(A) = x\Pi(aA),
\end{align*}
where the second equality is due to the induction hypothesis.  Next for any $X\in U(\h)$, we use the identity $\Pi(x\rhd A)+\Pi(A_1)(x\lhd A_2)=x\Pi(A)$ just proved to check that
\begin{align*}
&\,\, \quad\Pi(xX_1\rhd A_1)(X_2\lhd A_2)+\Pi(X_1\rhd A_1)(xX_2\lhd A_2)\\
&\stackrel{\eqref{eq:MP3}}{=} \Pi(x\rhd (X_1\rhd A_1))(X_2\lhd A_2)+
\Pi(X_1\rhd A_1)(x\lhd(X_2\rhd A_2))(X_3\lhd A_3)\\
&=x\Pi(X_1\rhd A_1)(X_2\lhd A_2)=xX\Pi(A),
\end{align*}
where the last equality is due to the induction hypothesis. The proof is completed.
\end{proof}

\vspace{0.1cm}
 \noindent
{\bf Acknowledgements}
This work is supported by National Natural Science Foundation of China (12071094, 12171155), and Basic and Applied Basic Research Foundation of Guangdong Province (2026A1515012750).

\bibliographystyle{amsplain}

\begin{thebibliography}{99}

\bibitem{AS} N. Andruskiewitsch and H.-J. Schneider, Pointed Hopf algebras. In Recent developments in Hopf algebras Theory, MSRI Publ. {\bf 43}
(2002), 1--68, Cambridge Univ. Press.

\bibitem{AGV} I. Angiono, C. Galindo and L. Vendramin, Hopf braces and Yang-Baxter operators, \textit{Proc. Amer. Math. Soc.} {\bf 145} (2017), 1981--1995.


\bibitem{BB} M. Beattie and D. Bulacu, Braided Hopf algebras obtained from coquasitriangular Hopf algebras, \textit{Comm. Math. Phys.} {\bf 282} (2008), 115--160.

\bibitem{Doi} Y. Doi, Unifying Hopf modules, J. Algebra, {\bf 153} (1992), 373--385.

\bibitem{DT} Y. Doi and M. Takeuchi, Multiplication alteration by two-cocycles\ -\ the quantum version, {\it Comm. Algebra} {\bf 22} (1994), 5715--5732.

\bibitem{D1} V. G. Drinfel'd, On almost cocommutative Hopf algebras, {\it Leningrad Math. J.} {\bf 1} (1990), 321--342.

\bibitem{ELM} K. Ebrahimi-Fard, A. Lundervold and H. Munthe-Kaas, On the Lie enveloping algebra of a post-Lie algebra, \textit{J. Lie Theory} {\bf 25} (2015), 1139--1165.

\bibitem{VRP}
J. M. Fern\'{a}ndez Vilaboa, R. Gonz\'{a}lez Rodr\'{\i}guez, B. Ramos P\'{e}rez, Twisted post-Hopf algebras, twisted relative Rota-Baxter operators and Hopf trusses, {\it SIGMA} {\bf 21} (2025), Page No. 024, 36 pp.

\bibitem{VRP1}
J. M. Fern\'{a}ndez Vilaboa, R. Gonz\'{a}lez Rodr\'{\i}guez, B. Ramos P\'{e}rez,
Relative Rota-Baxter operators, modules and projections,
{\it Czech. Math. J.} {\bf 75} (2025), 865--913.

\bibitem{FS}
D. Ferri and A. Sciandra, Matched pairs and Yetter-Drinfeld braces, \textit{Cand. J. Math.} (2025), first view; arXiv:2406.10009.


\bibitem{Go} M. Goncharov, Rota-Baxter operators on cocommutative Hopf algebras, \textit{J. Algebra} {\bf 582} (2021), 39--56.


\bibitem{GGV}
J. A. Guccione, J. J. Guccione and L. Vendramin, Yang-Baxter operators in symmetric categories, \textit{Comm. Algebra} {\bf46} (2018), 2811--2845.

\bibitem{GGV0}
J. A. Guccione, J. J. Guccione and C. Valqui, Set-theoretic type solutions of the braid equation, {\it J. Algebra} {\bf 644} (2024), 461--525.

\bibitem{GGV1}
J. A. Guccione, J. J. Guccione and C. Valqui, Set-theoretic type solutions of the braid equation, arXiv:2008.13494 (v5).

\bibitem{GLS} L. Guo, H. Lang and Y. Sheng, Integration and geometrization of Rota-Baxter Lie algebras, \textit{Adv. Math.} {\bf 387} (2021), 107834.

\bibitem{HN}
E. Habbestad and S. Neshveyev, Cocycle twisting of semidirect products and transmutation, {\it Int. Math. Res. Not.} {\bf 2024} (2024), 9142--9164.


\bibitem{JSZ}
J. Jiang, Y. Sheng and C. Zhu,
Lie theory and cohomology of relative Rota-Baxter operators,
\emph{J. Lond. Math. Soc.} \textbf{109} (2024), e12863.

\bibitem{Ku}
B. A. Kupershmidt, What a classical $r$-matrix really is, \emph{J. Nonlinear Math. Phys.} {\bf 6} (1999), 448--488.


\bibitem{Li0} Y. Li, Matched pairs and Yang-Baxter operators, arXiv:2501.11975v2.

\bibitem{Li} Y. Li, Matched pairs and Yang-Baxter operators, {\it Forum Math.} (2026), 13 pp, 
    https://doi.org/10.1515/forum-2025-0274.

\bibitem{LST} Y. Li, Y. Sheng and R. Tang, Post-Hopf algebras, relative Rota-Baxter operators and solutions to the Yang-Baxter equation, \textit{J. Noncommut. Geom.} {\bf 18} (2024), 605--630.


\bibitem{LYZ} J. Lu, M. Yan and Y. Zhu,   On the set-theoretical Yang-Baxter equation, \textit{Duke Math. J.} {\bf 104} (2000),   1--18.

\bibitem{Ma0} S. Majid, Crossed products by braided groups and bosonization, \textit{J. Algebra} {\bf 163} (1994), 165--190.

\bibitem{Ma1} S. Majid, Braided groups, {\it J. Pure Appl. Algebra} {\bf 86} (1993), 187--221.

\bibitem{Ma2} S. Majid, Doubles of quasitriangular Hopf algebras, \textit{Comm. Algebra} {\bf 19} (1991), 3061--3073.

\bibitem{Ma3} S. Majid, Physics for algebraists: Non-commutative and non-cocommutative Hopf algebras by a bicrossproduct construction, \textit{J. Algebra} {\bf 130} (1990), 17--64.

\bibitem{Maj1} S. Majid, A quantum groups primer, London Mathematical Society Lecture Note Series {\bf 292}, Cambridge University Press, Cambridge, 2002.

\bibitem{Ma} S. Majid, Foundations of quantum group theory, Cambridge University Press, 1995.

\bibitem{Mon} S. Montgomery, Hopf algebras and their actions on rings, Amer. Math. Soc., Regional Conf. Ser. in Math., \textbf{82}, 1993.

\bibitem{Ra} D. Radford, The structure of Hopf algebras with a projection, \textit{J. Algebra} {\bf 92} (1985), 322--347.


\bibitem{Sc}
A. Sciandra, Yetter-Drinfeld post-Hopf algebras and Yetter-Drinfeld relative Rota-Baxter operators, \textit{J. Noncommut. Geom.} {\bf 20} (2026), 995--1028.


\bibitem{STS}
M. Semenov-Tian-Shansky, What is a classical $r$-matrix? \emph{Funct. Anal. Appl.} {\bf 17} (1983), 259--272.

\bibitem{Ta} M. Takeuchi, Matched pairs of groups and bismash products of Hopf algebras, \textit{Comm. Algebra}  {\bf 9} (1981), 841--882.

\bibitem{TS} R. Tang and Y. Sheng, Leibniz bialgebras, relative Rota-Baxter operators and the classical Leibniz Yang-Baxter equation, {\it J. Noncommut. Geom.} {\bf 16} (2022), 1179--1211.

\bibitem{Uc} K. Uchino, Quantum analogy of Poisson geometry, related dendriform algebras and Rota-Baxter operators, {\it Lett. Math. Phys.} {\bf 85} (2008), 91--109.

\bibitem{Zh}
H. Zhu, Relative Yetter-Drinfeld modules and comodules over braided groups, {\it J. Math. Phys.} {\bf 56} (2015), 041706.

\bibitem{ZD}
H. Zhu and Y. Di, Cartier-Gabriel-Kostant theorem for relative Rota-Baxter operators, {\it J. Algebra} {\bf 685} (2026), 775--800.

\end{thebibliography}

\end{document}